\documentclass[11pt]{amsart}
\usepackage{a4wide,amssymb,color}
\usepackage[all]{xy}

\newtheorem{theorem}{Theorem}[section]
\newtheorem{lemma}[theorem]{Lemma}
\newtheorem{corollary}[theorem]{Corollary}

\newtheorem{proposition}[theorem]{Proposition}

\newtheorem{example}[theorem]{Example}
\newtheorem{problem}[theorem]{Problem}
\newtheorem{claim}[theorem]{Claim}

\theoremstyle{definition}

\numberwithin{equation}{section}

\newcommand{\e}{\varepsilon}
\newcommand{\w}{\omega}

\newcommand{\pr}{\mathrm{pr}}
\newcommand{\supp}{\mathsf{supp}}
\newcommand{\atom}{\mathsf{atom}}
\newcommand{\IR}{\mathbb R}
\newcommand{\NN}{\mathbb{N}}
\newcommand{\SM}{\setminus}
\newcommand{\xxx}{\mathbf x}

\newcommand{\A}{\mathcal A}
\newcommand{\Comp}{\mathbf{Comp}}
\newcommand{\IN}{\mathbb{N}}

\title[A simple Efimov space with sequentially-nice space of probability  measures]{A simple Efimov space with sequentially-nice space of probability  measures}
\author{Taras Banakh and Saak Gabriyelyan}

\address{Ivan Franko National University of Lviv (Ukraine) and Jan Kochanowski University in Kielce (Poland)}
\email{t.o.banakh@gmail.com}

\address{Department of Mathematics, Ben-Gurion University of the Negev, Beer-Sheva, P.O. 653, Israel}
\email{saak@math.bgu.ac.il}



\begin{document}

\begin{abstract}
Under Jensen's diamond principle $\diamondsuit$, we construct a simple Efimov space $K$ whose space of nonatomic probability measures  $P_{na}(K)$ is first-countable and sequentially compact. These two properties of $P_{na}(K)$ imply  that the space of probability measures $P(K)$ on $K$ is selectively sequentially pseudocompact and the Banach space $C(K)$ of continuous functions on $K$ has the Gelfand--Phillips property. We show also that any sequence of  probability measures on $K$ that converges to an atomic measure converges in norm, and any sequence of  probability measures on $K$ converging to zero in sup-norm has a subsequence converging to a nonatomic probability measure.
\end{abstract}

\maketitle

\section{Introduction}

For a Tychonoff space $K$, we denote by $P(K)$ the space of all regular probability measures on $K$ and by $C(K)$ the space of all real-valued continuous functions on $K$. Then the space $P(K)$ endowed with the weak$^\ast$ topology induced from the dual space $C(K)'$ of the Banach space $C(K)$ is a compact Hausdorff space.

The study of spaces of probability measures is an important direction of research in Measure Theory and  General Topology. For numerous results, open questions and historical remarks we refer the reader to the survey of Fedorchuk \cite{Fedorchuk-s} and the recent monograph of Bogachev \cite{Bog-18}. Many results in this area describe the interplay between  properties of the space $P(K)$ and  topological properties of the compact space $K$. For example, according to a classical result of Lindenstrauss \cite{Linden}, a compact space $K$ is Eberlein if and only if so is its space of probability measures $P(K)$. However, the question of whether $P(K)$ is  Corson compact for any Corson compact $K$ does not depend on ZFC.

The most important topological properties of compact spaces are properties of sequential type such as first-countability, sequential compactness, countable tightness or selective sequential pseudocompactness. Following \cite{DAS1}, we call a topological space $X$ {\em selectively sequentially pseudocompact} if for any sequence  $(U_n)_{n\in\w}$ of nonempty open sets in $X$ there exists a sequence $(x_n)_{n\in\w}\in\prod_{n\in\w}U_n$ containing a convergent subsequence. Clearly, every infinite compact selectively sequentially pseudocompact space contains many nontrivial convergent sequences. For any compact space $K$ we have the following implications.
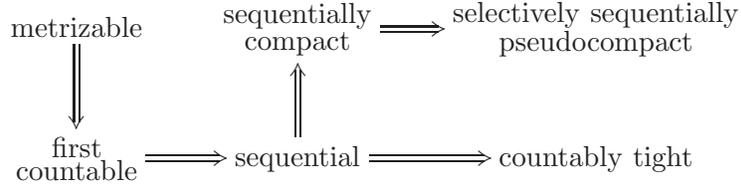
\[
\xymatrix{
\mbox{metrizable} \ar@{=>}[d]  & {\substack{\mbox{sequentially} \\ \mbox{compact}}} \ar@{=>}[r] &  {\substack{\mbox{selectively sequentially} \\ \mbox{pseudocompact}}} \\
 {\substack{\mbox{first} \\ \mbox{countable}}}  \ar@{=>}[r] &\mbox{sequential} \ar@{=>}[r] \ar@{=>}[u] &  \mbox{countably tight}
}
\]

It is easy to show that the space $P(K)$ is metrizable if and only if so is $K$. However, already the first-countability of $K$ may not imply even the countable tightness of $P(K)$. Using the Continuum Hypothesis Haydon \cite{Haydon-72} constructed a first-countable non-metrizable compact space $K$ which is the support of some nonatomic probability measure. As it was noticed by Fedorchuk \cite{Fedorchuk-s}, for the Haydon compact $K$ the tightness of $P(K)$ is uncountable. These results motivate the inverse question:  Assume that $P(K)$ has one of the sequential properties from the diagram. {\em What  can be said about sequential properties of the compact space $K$}? Clearly, if $P(K)$ is countably tight (sequentially compact or first countable) then so is $K$. But what about the case when the compact space $P(K)$ is selectively sequentially pseudocompact? Is then also $K$ selectively sequentially pseudocompact? In our main result, under Jensen's diamond principle $\diamondsuit$, we answer this question in the negative in a very strong form. More precisely, under $\diamondsuit$ we construct a simple Efimov space $K$ whose space of probability measures $P(K)$ contains a dense first-countable sequentially compact subspace and hence is selectively sequentially pseudocompact. Let us recall that a compact space is called {\em Efimov} if it  contains neither nontrivial convergent sequences nor  a topological copy of the Stone-\v{C}ech compactification $\beta\w$ of the discrete space $\w$ (we recall the notion of a simple compact space below in the section Preliminaries). All known examples of Efimov spaces are constructed only under some additional set-theoretic assumptions: CH \cite{Fed75,DM}, $\diamondsuit$ \cite{Fed76,DP}, $\mathfrak s=\w_1<2^{\w_1}=\mathfrak c$ \cite{Fed77} or $\mathrm{cf}([\mathfrak s]^\w)=\mathfrak s<2^{\mathfrak s}<2^{\mathfrak c}$ \cite{Dow}.  Yet, the existence of a Efimov space in ZFC is a major unsolved problem of General and Set-Theoretic Topology, see \cite{Nyikos,Hart}.
Our example of a Efimov space $K$ is based on a classical construction of a  Efimov space due to Fedorchuk and follows \cite{DP}.

Our interest to the selective sequential pseudocompactness is motivated by a result of Drew\-nowski and Emmanuele \cite[Theorem~4.1]{DrewEm} which states that if a compact space $K$ is  selectively sequentially pseudocompact (that is, $K$ belongs to the class $\mathcal{K}''$ in the terminology of \cite{DrewEm}), then the Banach space $C(K)$ has the Gelfand--Phillips property. Recall that a Banach space $E$ has the {\em Gelfand--Phillips property} if every limited set in $E$ is precompact, where a bounded subset $B$ of  $E$ is called {\em limited} if each weak$^\ast$ null sequence $(\chi_n)_{n\in\w}$ in the dual space $E^*$ converges uniformly on $B$, that is $\lim_n \sup\{ |\chi_n(x)|:x\in B\} =0$. In \cite{BG-GP-Banach},  we prove that the selective sequential pseudocompactness of $K$ in this theorem of Drewnowski and Emmanuele can be weakened to the selective sequential pseudocompactness of any subspace $A\subseteq P(K)$ that contains $K$. In particular, if $P(K)$ is selectively sequentially pseudocompact, then the Banach space $C(K)$ has the Gelfand--Phillips property.
Since any compact selectively sequentially pseudocompact space has many nontrivial convergent sequences, Drewnowski and Emmanuele  ask in \cite{DrewEm} whether there exists a compact space $K$ without nontrivial convergent sequences such that the Banach space $C(K)$ still has  the Gelfand--Phillips property. Under CH, this question was answered negatively by Schlumprecht \cite{Schlumprecht-Ph}.
These results motivate a problem posed in \cite[Problem 2.7]{BG-GP-Banach} of finding a Efimov space $K$ whose probability measure space $P(K)$ has in addition a quite strong sequential type property of being a selectively sequentially pseudocompact space. 

Let $K$ be a compact space. The Hahn decomposition theorem states that any probability measure $\mu\in P(K)$ can be represented as a convex combination
\[
\mu=t\mu_a+(1-t)\mu_{na},
\]
where $t\in[0,1]$, $\mu_a$ is atomic and $\mu_{na}$ is nonatomic. Denote by $P_a(K)$ and $P_{na}(K)$ subspaces of $P(K)$ consisting of all atomic and all nonatomic measures, respectively. It is clear that $\delta[K]\subseteq P_a(K)$ and $P_a(K)\cap P_{na}(K)=\emptyset$. This observation motivates the problem of study of (topological) properties of the spaces $P_{a}(K)$ and $P_{na}(K)$ and their relationships to (topological) properties of $K$ and $P(K)$.

The set $P_a(K)$ is dense in $P(K)$ and can be considered also as a subset of the Banach space $\ell_1(K)$ as well as a subset of the Banach space $c_0(K)$.
Since $\ell_1(K)$ has the Schur property, one can naturally ask: Can  also $P_a(K)$ have the same property in the sense that every convergent sequence in $P_a(K)$ is also norm convergent? On the other hand, it is interesting to understand what happen with sequences $(\mu_n)_{n\in\w}$ converging to zero in the sup-norm of $c_0(K)$. This motivates the following notion. A sequence $(\mu_n)_{n\in\w}$ of measures on $K$ (not necessary atomic) is defined to be {\em $c_0$-vanishing} if
\[
\lim_{n\to\infty} \sup_{x\in K} \mu(\{x\})=0.
\]
It turns out that  our Efimov space $K$ satisfies very strong conditions: (1) every sequence in $P(K)$ that converges to an atomic measure is norm convergent, and (2) every $c_0$-vanishing sequence in $P(K)$ has a subsequence converging to a nonatomic measure.
The latter property of $K$ implies that the space $P_{na}(K)$ is sequentially compact.  
In fact, our Efimov space $K$  is {\em crowded} (i.e., it has no isolated points), and therefore the  set $P_{na}(K)$ is dense in $P(K)$. Moreover, $P_{na}(K)$ is a first-countable sequentially compact space. The space $P_{na}(K)$ is not metrizable (and cannot be metrizable as every metrizable sequentially compact space is compact and $P_{na}(X)$ is not compact for a crowded compact space $X$).  On the other hand, in Example  \ref{exa:Pna-metrizable} we construct a non-metrizable crowded  compact space $C$ whose space of nonatomic probability measures $P_{na}(C)$ is metrizable.   

Now we formulate our main result in which (under $\diamondsuit$) the clause (iv) solves Problem 2.7 from \cite{BG-GP-Banach}.

\begin{theorem} \label{t:main}
Under Jensen's Diamond Principle $\diamondsuit$, there exists a simple crowded Efimov space $K$ such that
\begin{enumerate}
\item[{\rm(i)}] every $c_0$-vanishing sequence $(\mu_n)_{n\in\w}$ in $P(K)$ has a subsequence converging to a nonatomic measure $\mu\in P_{na}(K)$;
\item[{\rm(ii)}]  every sequence in $P(K)$ that converges to an atomic measure converges  in norm, so the weak$^\ast$ convergence and the norm convergence coincide  on $P_a(K)$;
\item[{\rm(iii)}]  the space $P_{na}(K)$ is non-metrizable, first-countable, \v{C}ech-complete, sequentially compact, and the set $P_{na}(K)$ is dense in $P(K)$;
\item[{\rm(iv)}] the space $P(K)$ is selectively sequentially pseudocompact but not sequentially compact;
\item[{\rm(v)}]  the Banach space $C(K)$ has the Gelfand--Phillips property.
\end{enumerate}
\end{theorem}
\noindent

Theorem~\ref{t:main} will be proved in Section~\ref{s:main} after some preliminary work done in Sections \ref{s:pre}--\ref{sec:enumeration}.


\section{Preliminaries}\label{s:pre}


All topological spaces in the paper are Tychonoff. A subset of a topological space is {\em clopen} if it is both closed and open. A topological space $X$ is called {\em crowded} if it has no isolated points. An indexed family of subsets $(F_i)_{i\in I}$ of a set is called {\em disjoint} if $F_i\cap F_j=\emptyset$ for any distinct indices $i,j\in I$. For two sets $A,B$ we write $A\subseteq^* B$ if $A\setminus B$ is finite.
If $f:X\to Y$ is a map and $F$ is a subset of $X$, we denote by $f[F]$ the image $\{f(x):x\in F\}$ of the set $F$. For a set $A$ and an indexed family $(x_\alpha)_{\alpha\in A}$ of points of a topological space $X$, we say that $(x_\alpha)_{\alpha\in A}$ converges to a point $x\in X$ and write $\lim_{\alpha\in A}x_\alpha=x$ if for every neighborhood $O_x\subseteq X$ the set $\{\alpha\in A:x_\alpha\notin O_x\}$ is finite.



The following lemma is Corollary 3.4 in \cite{DAS1}. We give its direct and short proof for the sake of completeness.
\begin{lemma} \label{l:dense-ssp}
If some dense subspace $Y$ of a topological space $X$ is selectively sequentially pseudocompact, then $X$ itself is selectively sequentially pseudocompact.
\end{lemma}

\begin{proof}
Let $\{U_n\}_{n\in\w}$ be  a sequence of nonempty open subsets of $X$. Then $\{Y\cap U_n\}_{n\in\w}$ is  a sequence of nonempty open subsets of $Y$. Since $Y$ is selectively sequentially pseudocompact, there exists a sequence $(x_n)_{n\in\w}\in\prod_{n\in\w}Y\cap U_n$ containing a convergent subsequence $(x_{n_k})_{k\in\w}$. It is clear that $(x_n)_{n\in\w}\in\prod_{n\in\w} U_n$ and $(x_{n_k})_{k\in\w}$ converges in $X$. Thus $X$ is selectively  sequentially pseudocompact.
\end{proof}

By a {\em measure} on a compact space $X$ we understand any regular $\sigma$-additive measure $\mu:\mathcal B(X)\to [0,\infty)$ on the $\sigma$-algebra $\mathcal B(X)$ of Borel subsets of $X$. The {\em regularity} of $\mu$ means that for every Borel set $B\subseteq X$ and every $\e>0$ there exists a closed subset $F\subseteq B$ of $X$ such that $\mu(B\setminus F)<\e$. A measure $\mu$ is a {\em probability measure} if $\mu(X)=1$. The space $P(X)$ of probability measures on $X$ is endowed with the topology
generated by the subbase consisting of the sets
$\{\mu\in P(X):\mu(U)>a\}$ where $a\in\IR$ and $U$ runs over open subsets of $X$.

By the Riesz Representation Theorem \cite[7.6.1]{Jar}, the space $P(X)$ can be identified with the subspace
\[
\{\mu\in C'(X):\|\mu\|=1=\mu(1_X)\}
\]
of the dual Banach space $C'(X)$, endowed with the weak$^\ast$ topology. Besides the weak$^\ast$ topology we shall also use the topology generated by the norm on $P(X)$. For two measures $\mu,\lambda$ on $X$ their norm-distance can be found by the formula
\[
\|\mu-\lambda\|=\sup_{B\in\mathcal B(X)}|\mu(B)-\lambda(B)|.
\]

Each element $x\in X$ can be identified with the Dirac measure
\[
\delta_x:\mathcal B(X)\to\{0,1\},\quad\delta_x:B\mapsto\begin{cases}1&\mbox{if $x\in B$};\\
0&\mbox{otherwise}.
\end{cases}
\]
So, $X$ can be identified with the closed subspace $\{\delta_x:x\in X\}$ of $P(X)$.

The convex hull of $X$ in $P(X)$ is denoted by $P_\w(X)$ and elements of $P_\w(X)$ are called {\em finitely supported measures} on $X$. The {\em support} of a measure $\mu\in P(X)$ denoted by $\supp(\mu)$ is the smallest closed subset $F\subseteq K$ such that $\mu(X\setminus F)=0$. The support of a finitely supported measure $\mu\in P_\w(X)$ is finite and coincides with the set
\[
\atom(\mu)=\{x\in X:\mu(\{x\})>0\}
\]
 of atoms of $\mu$. The additivity of a probability measure $\mu\in P(X)$ implies that the set $\atom(\mu)$ is at most countable and $\mu(\atom(\mu))\le 1$.
A measure $\mu\in P(X)$ is called
\begin{itemize}
\item {\em atomic} if $\mu(\atom(\mu))=1$;
\item {\em nonatomic} if $\atom(\mu)=\emptyset$.
\end{itemize}
By $P_a(X)$ and $P_{na}(X)$ we denote the subspaces of $P(X)$ consisting of all atomic measures and all nonatomic measures on $X$, respectively. Note that the set $P_a(X)$ of atomic measures contains $P_\w(X)$ and hence is dense in $P(X)$. 

Every continuous map $f:X\to Y$ between compact spaces induces a continuous map $Pf:P(X)\to P(Y)$ assigning to every $\mu\in P(X)$ the measure $Pf(\mu):\mathcal{B}(Y)\to[0,1]$ defined by the formula
\[
Pf(\mu):B\mapsto \mu\big(f^{-1}[B]\big) \quad (B\in\mathcal{B}).
\]

\begin{lemma}\label{l:Cech}
For every \textup{(}crowded\textup{)} compact space $X$, the set $P_\w(X)$ is dense in $P(X)$ and the set $P_{na}(X)$ is a \textup{(}dense\textup{)} $G_\delta$-set in $P(X)$.
\end{lemma}

\begin{proof}
The density of $P_\w(X)$ in $P(X)$ is a well-known fact, but for the convenience of the reader, we present here a simple proof. Fix any measure $\mu\in P(X)$ and a neighborhood $O_\mu$ of $\mu$ in $P(X)$. By the definition of the topology on $P(X)$, there exist $\e>0$ and open sets $U_1,\dots,U_n$ in $X$ such that
\[
\bigcap_{i=1}^n\{\lambda\in P(X):\lambda(U_i)>\mu(U_i)-\e\}\subseteq O_\mu.
\]
Let $\mathcal{B}$ be the smallest Boolean algebra of subsets of $X$ containing the family $\{U_1,\dots,U_n\}$, and let $\A$ be the set of atoms of $\mathcal{B}$. In each set $A\in\A$ choose a point $x_A\in A$ and observe that the finitely supported measure $\sum_{A\in\A}\mu(A)\cdot\delta_{x_A}$ belongs to the intersection $O_\mu\cap P_\w(X)$, witnessing that the set $P_\w(X)$ is dense in $P(X)$.

Next, we show that the set $P_{na}(X)$ is (dense) $G_\delta$ in $P(X)$.  For every natural number $n>0$, consider the continuous map
\[
\Xi_n:P(X)\times X\times[\tfrac{1}{n},1]\to P(X),\quad \Xi_n:(\mu,x,t)\mapsto (1-t)\mu+t\delta_x,
\]
 and observe that the image
\[
A_n=\Xi\big(P(X)\times X\times[\tfrac{1}{n},1]\big)=\big\{\mu\in P(X):\exists x\in X\;\;\mu(\{x\})\ge \tfrac{1}{n}\big\}
\]
is compact and hence closed in $P(X)$. Since $P_{na}(X)=\bigcap_{n=1}^\infty(P(X)\setminus A_n)$, the set $P_{na}(X)$ is of type $G_\delta$ in $P(X)$.

If the space $X$ is crowded, then every set $A_n$ is nowhere dense in $P(X)$. Indeed, given any nonempty open set $U\subseteq P(X)$ we can use the density of $P_\w(X)$ in $P(X)$ and find a finitely supported measure $\mu\in U\cap P_\w(X)$. By the definition of the topology on $P(X)$, there exist a positive real number $\e>0$ and a  family of open sets $(O_x)_{x\in\supp(\mu)}$ in $X$ such that $x\in O_x$ for every $x\in\supp(\mu)$ and
\[
\bigcap_{x\in\supp(\mu)}\{\lambda\in P(X):\lambda(O_x)>\mu(O_x)-\e\}\subseteq U.
\]
Since the space $X$ is Hausdorff, we can additionally assume that the indexed family $(O_x)_{x\in\supp(\mu)}$ is disjoint. The absence of isolated points in the crowded space $X$ ensures that for every $x\in\supp(\mu)$, the set $O_x$ is infinite and hence contains a finite subset $F_x$ of cardinality $|F_x|>n$. Then the finitely supported measure
\[
\sum_{x\in\supp(\mu)}\sum_{y\in F_x}\tfrac{\mu(\{x\})}{|F_x|}\delta_y
\]
belongs to the set $U\setminus A_n$, witnessing that the closed set $A_n$ is nowhere dense in $P(X)$.

By the Baire Theorem, the $G_\delta$-set $P_{na}(X)=\bigcap_{n\in\IN}(P(X)\setminus A_n)$ is dense in $P(X)$, being a countable intersection of open dense sets.
\end{proof}

\begin{example} \label{exa:Pna-metrizable}
There is a crowded compact nonmetrizable space $C$ such that the space $P_{na}(C)$ is metrizable.
\end{example}

\begin{proof}
Let $D$ be an uncountable discrete space. Denote by $C=\{\infty\} \cup (D\times 2^{\w})$ the one point compactification of the space $D\times 2^{\w}$. It is clear that $C$ is  crowded and non-metrizable, and $D\times 2^\w$ is a metrizable open dense subspace in $C$. By Theorem 1.4 in \cite{Ban1} and Theorem 4.4 in \cite{Ban2}, the subspace $\hat P(D\times 2^\w)=\{\mu\in P(C):\mu(D\times 2^\w)=1\}$ of $P(C)$ is metrizable and so is its subspace $P_{na}(C)\subset \hat P(D\times 2^\w)$.
\end{proof}

A compact space $X$ is {\em simple} if it is the limit of a well-ordered continuous spectrum $\{X_\alpha,\pi_\alpha^\beta:\alpha\le\beta<\w_1\}$ of length $\w_1$ consisting of zero-dimensional compact metrizable spaces $X_\alpha$ and simple bonding maps $\pi_{\alpha}^{\alpha+1}$. A map $f:X\to Y$ between topological spaces is called {\em simple} if there exists $y\in Y$ such that $1\le |f^{-1}(y)|\le 2$, $f^{-1}(y)$ is nowhere dense in $X$ and $|f^{-1}(z)|=1$ for every $z\in Y\setminus\{y\}$. For more information on inverse spectra, see \cite[\S2.5]{Eng}.

According to \cite{Kop,Dow}, simple compact spaces do not contain topological copies of $\beta\w$, so constructing a simple Efimov space $X$ one should only care about killing all convergent sequences in $X$. Another helpful property of simple compact spaces is the uniform regularity of nonatomic probability measures on such spaces, see \cite{DP}.

A measure $\mu\in P(X)$ on a compact space $X$ is called {\em uniformly regular} if there exists a continuous map $f:X\to Y$ to a compact metrizable space $Y$ such that $\mu(F)=\mu\big(f^{-1}[f[F]]\big)$ for every closed set $F\subseteq X$. 

\begin{lemma}\label{l:first-countable}
For every simple compact space $X$, the space $P(X)$ is first-countable at points of the set $P_{na}(X)$.
\end{lemma}

\begin{proof}
By Lemma~4.1  of \cite{DP},  any nonatomic measure on a simple compact space is uniformly regular. By (the proof of) Proposition~2 in \cite{Pol} (see also the statement (v) in \cite[p.2605]{DP}) the space $P(X)$ is first countable at each uniformly regular measure. Consequently, $P(X)$ is first-countable at points of the set $P_{na}(X)$.
\end{proof}


\section{Measures on sequentially discrete compact spaces}


A topological space $X$ is defined to be {\em sequentially discrete} if every convergent sequence in $X$ is eventually constant. So, $X$ is sequentially discrete if and only if $X$ contains no non-trivial convergent sequences.

We shall use the following probabilistic characterization of sequentially discrete compact spaces which can have an independent value.

\begin{proposition} \label{p:compact-no-sequence}
For a compact space $K$ the following assertions are equivalent:
\begin{enumerate}
\item[{\rm(i)}] $K$ is sequentially discrete;
\item[{\rm(ii)}] if a sequence $(\mu_n)_{n\in\w}\in P(K)^\w$ converges to a measure $\mu\in P(K)$, then
\[
\limsup_{n\to\infty}\sup_{x\in K}\big(\mu_n(\{x\})-\mu(\{x\})\big)\leq 0.
\]
\end{enumerate}
\end{proposition}

\begin{proof}
(i)$\Rightarrow$(ii) Suppose for a contradiction that there is  a sequence $(\mu_n)_{n\in\w}\in P(K)^\w$ converging to some measure $\mu\in P(K)$ such that
\begin{equation} \label{equ:compact-non-0}
\limsup_{n\to\infty}\sup_{x\in K}\big(\mu_n(\{x\})-\mu(\{x\})\big)> 0.
\end{equation}
Let $A$ be the set of all numbers $c\in (0,1]$ for which  there is  a sequence $(\mu_n)_{n\in\w}\in P(K)^\w$ converging to some measure $\mu\in P(K)$ such that
\[
\limsup_{n\to\infty}\sup_{x\in K}\big(\mu_n(\{x\})-\mu(\{x\})\big)\geq c.
\]By our assumption, $A$ is not empty and hence  $A$ has the least upper bound $\sup(A)\in (0,1]$. It is clear that $(0,\sup(A))\subseteq A$. Therefore the inequality $\tfrac{\sup(A)}{1-\sup(A)} >\sup(A)$ implies that there exist elements $a,a'\in A\cap (0,1)$ such that $\tfrac{a}{1-a}>\sup(A)$ and $a<a'$. Choose  a sequence $(\mu_n)_{n\in\w}\in P(K)^\w$ converging to a measure $\mu\in P(K)$ such that
\begin{equation} \label{equ:compact-non-11}
\lim_{n\to\infty}\sup_{x\in K}\big(\mu_n(\{x\})-\mu(\{x\})\big)\geq a'>a
\end{equation}
Passing to a subsequence, we can assume additionally that
\begin{equation} \label{equ:compact-non-1}
\inf_{n\in\w}\sup_{x\in K}\big(\mu_n(\{x\})-\mu(\{x\})\big)> a.
\end{equation}


Consider the sets $F=\{x\in K:\mu(\{x\})\geq a\}$ and $F_n=\{x\in K:\mu_{n}(\{x\})\geq a\}$ for $n\in\w$. The additivity of the measure $\mu$ implies that these sets are finite of  cardinality $\leq 1/a$. The condition~(\ref{equ:compact-non-1}) ensures that for every $n\in\w$, the set $F_n$ is not empty.

\begin{claim}
The union $C:=F\cup\bigcup_{n\in\w} F_n$ is a compact subset of $K$.
\end{claim}

\begin{proof}
Indeed, since $F$ is finite, to prove the claim it suffices to show that every open neighborhood $U\subseteq K$ of the set $F$ contains all but finitely many sets $F_n$. Observe that for every $x\in K\setminus U$, we have $\mu(\{x\})<a$ and, by the regularity of $\mu$ and $K$, there exists a closed neighborhood $N_x\subseteq K$ of $x$ such that $\mu(N_x)<a$ and $N_x\cap F=\emptyset$. By the compactness of $K\setminus U$, there exist a finite subset $E\subseteq K\setminus U$ such that $K\setminus U\subseteq\bigcup_{x\in E}{N_x}$. Then the  set
\[
\mathcal{O}_\mu=\{\lambda\in P(K): \lambda(K\setminus N_x)>1-a \; \mbox{ for all }\; x\in E\}
\]
is an open neighborhood of $\mu$ in $P(K)$. Since the sequence $(\mu_n)_{n\in\w}$ converges to $\mu$, there exists $m\in\w$ such that $\mu_k\in \mathcal{O}_\mu$ for all $k\geq m$. We show that $F_k\subseteq U$ for every $k\geq m$. Indeed, assuming that $F_k\not\subseteq U$ for some $k\geq m$, we can find a point $z\in F_k\setminus U\subseteq K\setminus U$. Since $\{N_x:x\in E\}$ covers $K\setminus U$, there exists $x\in E$ such that $z\in N_x$.
It follows from $\mu_k\in \mathcal{O}_\mu$ that $\mu_k(K\setminus N_x)>1-a$ and hence $\mu_k(\{z\})\leq \mu_k(N_x)<a$, which contradicts the definition of the set $F_k\ni z$. This contradiction shows that every neighborhood of $F$ contains all but finitely many sets $F_n$. Thus $C$ is compact.
\end{proof}
\smallskip

Being countable, the compact space $C$ is metrizable. Since the space $K$ is sequentially discrete, the compact metrizable subspace $C$ of $K$ is discrete and hence finite. Therefore $F$ is an open subset of $C$, and hence $F_n\subseteq F$ for all sufficiently large numbers $n$. By the Pigeonhole Principle, for some subset $F'\subseteq F$ the set $J=\{n\in\w:F_n=F'\}$ is infinite.
Choose any point $z\in F'$. Since $\mu(\{z\})\ge a$, we can find a unique measure $\mu'\in P(K)$ such that $\mu=a\delta_z+(1-a)\mu'$. By the same reason, for every $n\in J$, there exists a unique measure $\mu'_n\in P_a(K)$ such that $\mu_n=a\delta_z+(1-a)\mu_n'$.
\smallskip

The convergence $(\mu_n)_{n\in\w}\to \mu$ implies the convergence $(\mu_n')_{n\in J}\to \mu'$. For every $n\in J$, (\ref{equ:compact-non-1}) implies
\[
\sup_{x\in K}\big(\mu'_n(\{x\})-\mu'(\{x\})\big) =\sup_{x\in K}\Big(\tfrac1{1-a}\big(\mu_n(\{x\})-a\delta_z(\{z\})-\mu(\{x\})+a\delta_z(\{z\})\Big)\geq \tfrac{a}{1-a}.
\]
Therefore, the number $\tfrac{a}{1-a}$ belongs to $A$ which is impossible since $\tfrac{a}{1-a}>\sup(A)$. This contradiction finishes the proof of the implication (i)$\Rightarrow$(ii).
\smallskip

(ii)$\Rightarrow$(i) If $K$ is not sequentially discrete, then $K$ contains a sequence of points $(z_n)_{n\in\w}$ that converge to some point $z\in K\setminus\{z_n\}_{n\in\w}$. Then the sequence of Dirac measures $(\delta_{z_n})_{n\in\w}$ converges to $\delta_z$ in the space $P(K)$, however
\[
\limsup_{n\to\infty}\sup_{x\in K}\big(\delta_{z_n}(\{x\})-\delta_z(\{x\})\big)\ge\limsup_{n\to\infty}\big(\delta_{z_n}(\{z_n\})-\delta_x(\{z_n\}\big)=1\not\le 0.
\]
\end{proof}

We recall that a sequence of measures $(\mu_n)_{n\in\w}$ on a Tychonoff space $X$ is {\em $c_0$-vanishing} if
\[
\lim_{n\to\infty}\sup_{x\in X}\mu_n(\{x\})=0.
\]
The following lemma will be used in the proof of our main theorem.

\begin{lemma} \label{cl:converge3}
Let $X$ be a sequentially discrete compact space such that every $c_0$-vanishing sequence $(\lambda_n)_{n\in\w}$ of probability measures on $X$  has a subsequence that converges to a nonatomic measure on $X$. Then for any sequence $(\mu_n)_{n\in\w}\in P(X)^\w$ that converges to a measure $\mu\in P(X)$, we have $\lim_{n\to\infty}\sup_{x\in X}|\mu_n(\{x\})-\mu(\{x\})|=0$.
\end{lemma}

\begin{proof}
To derive a contradiction, assume that  $\lim_{n\in\w}\sup_{x\in X}|\mu_n(\{x\})-\mu(\{x\})|\ne 0$ for some sequence $(\mu_n)_{n\in\w}\in P(K)^\w$ that converges to a measure $\mu\in P(K)$. Replacing $(\mu_n)_{n\in\w}$ by a suitable subsequence, we can additionally assume that
\begin{itemize}
\item for every $x\in X$, the limit $\mu_\infty(x)=\lim_{n\in\w}\mu_n(\{x\})$ exists, and
\item the limit $a=\lim_{n\in\w}\sup_{x\in X}|\mu_n(\{x\})-\mu(\{x\})|$ exists and hence is strictly positive.
\end{itemize}

\begin{claim}\label{cl:new1}
For every $x\in X$ we have $\mu_\infty(x)\le\mu(\{x\})$.
\end{claim}

\begin{proof}
Suppose for a contradiction that $\mu_\infty(x)>\mu(\{x\})$ for some $x\in X$. Choose any real number $a$ with $\mu_\infty(x)>a>\mu(\{x\})$. Then $\mu(X\setminus\{x\})=1-\mu(\{x\})>1-a$, and hence the set
\[
U=\{\lambda\in P(X): \lambda(X\setminus\{x\})>1-a\}
\]
is an open neighborhood of the measure $\mu$ in $P(X)$. Since $\lim_{n\in\w}\mu_n=\mu$ and $\lim_{n\in\w}\mu_n(\{x\})=\mu_\infty(x)>a$, there exists $n\in\w$ such that $\mu_n(\{x\})>a$ and $\mu_n\in U$. The latter inclusion implies that $\mu_n(X\setminus\{x\})>1-a$ and hence $\mu_n(\{x\})=1-\mu_n(X\setminus\{x\})<1-(1-a)=a$, which is a desired contradiction completing the proof of the claim.
\end{proof}

Claim~\ref{cl:new1} ensures that $\sum_{x\in X}\mu_\infty(x)\le\sum_{x\in X}\mu(\{x\})\le 1$ which implies that
\[
\mu_\infty=\sum_{x\in X}\mu_\infty(x)\delta_x
\]
is a well-defined atomic measure on $X$. For every $n\in\w$, consider the  measures
\[
\mu_n^+=\max\{\mu_n-\mu_\infty,0\}\quad\mbox{and}\quad
\mu_n^-=\max\{\mu_\infty-\mu_n,0\}
\]
and observe that $\mu_n-\mu_\infty=\mu_n^+-\mu_n^-$.

By Claim~\ref{cl:new1} and Proposition~\ref{p:compact-no-sequence}, we have
\[
\begin{aligned}
0\le \lim_{n\in\w}\sup_{x\in X}\mu_n^+(\{x\})& =\lim_{n\in\w}\sup_{x\in X}\max\{(\mu_n(\{x\})-\mu_\infty(\{x\}),0\}\\
& \le\lim_{n\in\w}\sup_{x\in X}\max\{(\mu_n(\{x\})-\mu(\{x\}),0\}= 0,
\end{aligned}
\]
which means that the sequence of measures $(\mu_n^+)_{n\in\w}$ in $c_0$-vanishing.

\begin{claim}
The sequence $(\mu_n^-)_{n\in\w}$ is $c_0$-vanishing.
\end{claim}

\begin{proof}
Since $\sum_{x\in X}\mu_\infty(x)\le 1$, for every $\e>0$, there exists a finite set $F\subseteq X$ such that $\sup_{x\in X\setminus F}\mu_\infty(x)<\e$. Since, for every $x\in F$, the sequence $(\mu_n(\{x\}))_{n\in\w}$ converges to $\mu_\infty(x)$, there exists $m\in\w$ such that $\sup_{n\ge m}\max_{x\in F}|\mu_n(\{x\})-\mu_\infty(x)|<\e$. Then for every $n\ge m$ and every $x\in F$, we have
\[
\mu_n^-(\{x\})=\max\{0,\mu_\infty(x)-\mu_n(\{x\})\}\le|\mu_n(\{x\})-\mu_\infty(x)|<\e.
\]
Also for every $x\in X\setminus F$ we have
\[
\mu_n^-(\{x\})=\max\{0,\mu_\infty(x)-\mu_n(\{x\})\}\le\mu_\infty(\{x\})<\e.
\]
Therefore, $\sup_{n\ge m}\sup_{x\in X}\mu^-_n(\{x\})\le\e$ and the sequence $(\mu_n^-)_{n\in\w}$ is $c_0$-vanishing.
\end{proof}

\begin{claim}\label{cl:new5}
Each norm-bounded $c_0$-vanishing sequence of measures $(\lambda_n)_{n\in\w}$ on $X$ contains a subsequence that converges to a nonatomic measure.
\end{claim}

\begin{proof}
We lose no generality assuming that $\lambda_n(X)\ne 0$ for every $n\in\w$. Since $\sup_{n\in\w}\lambda_n(X)=\sup_{n\in\w}\|\lambda_n\|<\infty$, we can find an infinite subset $I\subseteq \w$ such that the sequence $(\lambda_n(X))_{n\in I}$ converges to some real number $a\ge 0$. If $a=0$, then $(\lambda_n(X))_{n\in I}$ converges to the zero measure, which is nonatomic.

If $a>0$, then $\big(\tfrac{\lambda_n}{\|\lambda_n\|})_{n\in I}$ is a $c_0$-vanishing sequence of atomic probability measures. By the assumption of the lemma, there exists an infinite set $J\subseteq I$ such that the subsequence $\big(\frac{\lambda_n}{\|\lambda_n\|}\big)_{n\in J}$ converges to some nonatomic probability measure $\lambda_\infty$. Since $\|\lambda_n\|\to a\ne 0$, it follows  that the sequence $(\lambda_n)_{n\in J}$ converges to the nonatomic measure $a\lambda_\infty$.
\end{proof}

By Claim~\ref{cl:new5},  for the $c_0$-vanishing sequence $(\mu^+_n)_{n\in\w}$ there exists an infinite set $I^+\subseteq \w$ such that the subsequence $(\mu^+_n)_{n\in I^+}$ converges to some nonatomic measure $\mu^+$. Since the sequence $(\mu_n^-)_{n\in I^+}$ is $c_0$-vanishing, there exists an infinite set $I\subseteq I^+$ such that the sequence $(\mu_n^-)_{n\in I}$ converges to some nonatomic measure $\mu^-$ on $X$. Then the sequence $(\mu_n-\mu_\infty)_{n\in I}=(\mu_n^+-\mu_n^-)_{n\in I}$ converges to the nonatomic sign-measure $\mu^+-\mu^-$. On the other hand, this sequence converges to the measure $\mu-\mu_\infty$. Therefore, the measure $\mu-\mu_\infty=\mu^+ -\mu^-$ is nonatomic and hence
$$
\mu(\{x\})=\mu_\infty(\{x\})\quad \mbox{ for every } \; x\in X.
$$

Now the $c_0$-vanishing property of the sequences $(\mu_n^+)_{n\in\w}$ and $(\mu_n^-)_{n\in\w}$ guarantees that
\[
\begin{aligned}
\limsup_{n\in \w}\sup_{x\in X}|\mu_n(\{x\})-\mu(\{x\})| & {=}\limsup_{n\in\w}\sup_{x\in X}\big|\mu_n(\{x\})-\mu_\infty(\{x\})\big|\\
& =\limsup_{n\in\w}\sup_{x\in X}\max\{\mu_n^+(\{x\}),\mu_n^-(\{x\})\}=0.
\end{aligned}
\]
\end{proof}

\begin{lemma} \label{l:Pna-norm}
Let $X$ be a sequentially discrete compact space such that every $c_0$-vanishing sequence $(\lambda_n)_{n\in\w}$ of probability measures on $X$  has a subsequence that converges to a nonatomic measure on $X$. Then any sequence $(\mu_n)_{n\in\w}\in P(X)^\w$ that converges to an atomic measure $\mu\in P_a(X)$ converges to $\mu$ in norm.
\end{lemma}

\begin{proof}
To show that $\lim_{n\in\w}\|\mu_n-\mu\|=0$, fix any $\e>0$. Since $\|\mu\|=\sum_{x\in \atom(\mu)}\mu(\{x\})=1$, there exists a finite set $F\subseteq\atom(\mu)$ such that $\sum_{x\in F}\mu(\{x\})>1-\tfrac{1}{4}\e$. Let $\mu=\lambda+\nu$ where $\lambda=\sum_{x\in F}\mu(\{x\})\delta_x$ and $\nu=\mu-\lambda$. It follows that $\|\lambda\|=\mu(F)=\sum_{x\in F}\mu(\{x\})>1-\tfrac{1}{4}\e$ and hence $\|\nu\|=\mu(X\setminus F)<\frac{1}{4}\e$.

By Lemma \ref{cl:converge3}, $\lim_{n\in\w}\sup_{x\in X}|\mu_n(\{x\})-\mu(\{x\})|=0$, and hence there exists $m\in\w$ such that
\[
\sup_{n\ge m}\max_{x\in F}|\mu_n(\{x\})-\mu(\{x\})|<\tfrac{\e}{4|F|} \quad \mbox{ for every $n\geq m$}.
\]
For every $n\ge m$, write the measure $\mu_n$ as $\mu_n=\lambda_n+\nu_n$ where $\lambda_n=\sum_{x\in F}\mu_n(\{x\})\delta_x$ and $\nu_n=\mu_n-\lambda_n$. Observe that
$$\|\lambda_n\|=\mu_n(F)=\sum_{x\in F}\mu_n(\{x\})>\sum_{x\in F}\big(\mu(\{x\})-\tfrac{\e}{4|F|}\big)=\mu(F)-\tfrac14\e>1-\tfrac14\e-\tfrac14\e=1-\tfrac12\e$$and hence $\|\nu_n\|=\nu_n(X\setminus F)<\tfrac12\e$.
Then for every $n\geq m$, we have
\[
\begin{aligned}
\|\mu_n-\mu\| & =\|\lambda_n+\nu_n-\lambda-\nu\|\le\|\nu_n\|+\|\nu\|+\|\lambda_n-\lambda\|\\
& =\|\nu_n\|+\|\nu\|+\sum_{x\in F}|\mu_n(\{x\})-\mu(\{x\})|<\tfrac{1}{2}\e+\tfrac{1}{4}\e+|F|\cdot \tfrac{1}{4|F|}\e=\e,
\end{aligned}
\]
which means that $\lim_{n\in\w}\|\mu_n-\mu\|=0$.
\end{proof}

In the proof of Theorem~\ref{t:main} we shall need some results on the preservation of the $c_0$-vanishing property by projections $\pi^{\w_1}_\alpha:2^{\w_1}\to 2^\alpha$, $\pi^{\w_1}_\alpha:x\mapsto x{\restriction}_\alpha$. For every $\alpha\in\w_1$, the projection $\pi^{\w_1}_\alpha$ induces a continuous map $P\pi^{\w_1}_\alpha:P(2^{\w_1})\to P(2^\alpha)$ between the corresponding spaces of probability measures.

\begin{lemma}\label{l:atomic}
For any nonatomic measure $\mu\in P_{na}(2^{\w_1})$ there exists an ordinal $\alpha\in \w_1$ such that for any ordinal $\beta\in[\alpha,\w_1)$ the measure $P\pi^{\w_1}_\beta(\mu)\in P(2^\beta)$ is nonatomic.
\end{lemma}

\begin{proof}
Since $\mu$ is nonatomic and regular, for every $n\in\w$ and $x\in 2^{\w_1}$, there exists a neighborhood $O_{x,n}\subseteq 2^{\w_1}$ of $x$ such that $\mu(O_{x,n})<\frac1{2^n}$. We lose no generality assuming that $O_{x,n}$ is of basic form $\{y\in 2^{\w_1}:y{\restriction}_{F_{x,n}}=x{\restriction}_{F_{x,n}}\}$ for some finite set $F_{x,n}\subseteq\w_1$. By the compactness of $2^{\w_1}$ for every $n\in\w$, there exists a finite set $E_n\subseteq 2^{\w_1}$ such that $2^{\w_1}=\bigcup_{x\in E_n}O_{x,n}$. We claim that any countable ordinal $\alpha$ containing the countable set $\bigcup_{n\in\w}\bigcup_{x\in E_n}F_{x,n}$ has the desired property: for any ordinal $\beta\in[\alpha,\w_1)$ the measure $\lambda=P\pi^{\w_1}_\beta(\mu)$ is nonatomic. To derive a contradiction, assume that $\lambda(\{a\})>0$ for some $a\in 2^\beta$. Find $n\in\w$ such that $\frac1{2^n}<\lambda(\{a\})$. Choose any point $b\in 2^{\w_1}$ such that $\pi^{\w_1}_\beta(b)=a$ and find $x\in E_n$ such that $b\in O_{x,n}$. Then
\[
(\pi^{\w_1}_\beta)^{-1}(\{a\})=\{z\in 2^{\w_1}:z{\restriction}_\beta=b{\restriction}_\beta\}\subseteq \{z\in 2^{\w_1}:z{\restriction}_{F_{x,n}}=b{\restriction}_{F_{x_n}}\}=O_{x,n}
\]
and hence
\[
\lambda(\{a\})=\mu\big(\big(\pi^{\w_1}_\beta\big)^{-1}(\{a\})\big)\subseteq\mu(O_{x,n})<\tfrac1{2^n},
\]
which contradicts the choice of $n$.
\end{proof}

\begin{lemma}\label{l:vanish}
For every $c_0$-vanising sequence of measures $(\mu_n)_{n\in\w}\in\big(P(2^{\w_1})\big)^\w$ there exists a countable ordinal $\alpha$ such that for any ordinal $\beta\in[\alpha,\w_1)$ the sequence of measures $(P\pi^{\w_1}_\beta(\mu_n)\big)_{n\in\w}$ is $c_0$-vanishing.
\end{lemma}

\begin{proof} Write each measure $\mu_n$ as $\mu_n=(1-t_n)\lambda_n+t_n\nu_n$ for some $t_n\in[0,1]$ and measures $\lambda_n\in P_a(2^{\w_1})$ and $\nu_n\in P_{na}(2^{\w_1})$. By Lemma~\ref{l:atomic}, there exists a countable ordinal $\alpha\in\w_1$ such that for every $\beta\in[\alpha,\w_1)$ and every $n\in\w$ the measure $P\pi^{\w_1}_\beta(\nu_n)$ is nonatomic. Replacing $\alpha$ by a larger countable ordinal, if necessary, we can additionally assume that the projection $\pi^{\w_1}_\alpha$ is injective on the countable set $A=\bigcup_{n\in\w}\atom(\lambda_n)$.

We claim that the ordinal $\alpha$ has the required property. Indeed, fix any ordinal $\beta\in[\alpha,\w_1)$ and for every $n\in\w$ denote the measure $P\pi^{\w_1}_\beta(\mu_n)$ by $\mu_n^\beta$. It is clear that $\mu_n^\beta=(1-t_n)\lambda_n^\beta+t_n\nu_n^\beta$ where $\lambda_n^\beta=P\pi^{\w_1}_\beta(\lambda_n)$ and $\nu_n^\beta=P\pi^{\w_1}_\beta(\nu_n)$. The choice of the ordinal $\alpha$ ensures that the measure $\nu_n^\beta$ is nonatomic. The injectivity of the projection $\pi^{\w_1}_\alpha$ on $A$ implies the injectivity of the projection $\pi^{\w_1}_\beta$ on $A$. Then we can choose a function $f:2^\beta\to 2^{\w_1}$ such that $\pi^{\w_1}_\beta\circ f(y)=y$ for every $y\in 2^\beta$ and $f\circ \pi^{\w_1}_\beta(x)=x$ for every $x\in A$.

Observe that for every $y\in \pi^{\w_1}_\beta[A]$, we have
\[
\mu_n^\beta(\{y\})=(1-t_n)\lambda_n^\beta(\{y\})+t_n\nu_n^\beta(\{y\})=(1-t_n)\lambda_n(\{f(y)\})+0=\mu_n(\{f(y)\})
\]
and for every $y\in 2^\beta\setminus\pi^{\w_1}_\beta[A]$, we have
\[
\mu_n^\beta(\{y\})=(1-t_n)\lambda_n^\beta(\{y\})+t_n\nu_n^\beta(\{y\})=0+0.
\]
This implies
\[
\sup_{y\in 2^\beta}\mu_n^\beta(\{y\})=\sup_{x\in 2^{\w_1}}\mu_n(\{x\})
\]
 and hence
\[
\lim_{n\to\infty}\sup_{y\in 2^\beta}\mu_n^\beta(\{y\})=\lim_{n\to\infty}\sup_{x\in 2^{\w_1}}\mu_n(\{x\})=0,
\]
which means that the sequence $(P\pi^{\w_1}_\beta(\mu_n))_{n\in\w}= (\mu_n^\beta)_{n\in\w}$ is $c_0$-vanishing.
\end{proof}


\section{A key Lemma}


In the following lemma we use some notations related to the binary tree $2^{<\w}:=\bigcup_{n\in\w}2^n$. Here $2$ stands for the ordinal $2=\{0,1\}$. The unique element $\emptyset$ of $2^0$ is the root of the tree $2^{<\w}$. For an element $s\in 2^{<\w}$, we denote  by $|s|$ the unique number $n\in\w$ such that $s\in 2^n$. The number $|s|$ is called the {\em length} of $s$. For two elements $s,t\in 2^{<\w}$, we write $s\le t$ if $|s|\le|t|$ and $s=t{\restriction}_{|s|}$. Also we write $s<t$ if $s\le t$ and $s\ne t$. Let ${\uparrow}s:=\{t\in 2^{<\w}:s\le t\}$ be the upper set of $s$ in the tree $2^{<\w}$ (so ${\uparrow}s$ consists of all possible extensions of the function $s$). For $i\in 2$, let $s\hat{\;}i$ be the unique element of ${\uparrow}s$ such that $|s\hat{\;}i|=|s|+1$ and $s\hat{\;}i(|s|)=i$. So, $s\hat{\;}0$ and $s\hat{\;}1$ are immediate successors of $s$ in the tree $2^{<\w}$.

\begin{lemma}\label{l:atom}
Let $X$ be a zero-dimensional compact metrizable space, $(\mu_n)_{n\in\w}\in P(X)^\w$ be a $c_0$-vanishing sequence of probability measures that converge to a measure $\mu\in P(X)$. For every $z\in \atom(\mu)$ there exist a subsequence $(\mu_{n_k})_{k\in\w}$ of the sequence $(\mu_n)_{n\in\w}$ and a family $(X_s)_{s\in 2^{<\w}}$ of clopen subsets of $X\setminus\{z\}$ such $X_\emptyset=X\setminus\{z\}$ and for every $s\in 2^{<\w}$ the following conditions are satisfied:
\begin{enumerate}
\item[\textup{(i)}] $X_{s\hat{\;}0}\cup X_{s\hat{\;}1}=X_s$ and $X_{s\hat{\;}0}\cap X_{s\hat{\;}1}=\emptyset$;
\item[\textup{(ii)}] for any clopen neighborhood $U\subseteq X$ of $z$ and any $\e>0$ there exists $q\in\w$ such that
\[
\mu_{n_k}(U\cap X_s)\le \mu(U\setminus\{z\})+\tfrac{1}{2^{|s|}}+\e \quad \mbox{ for any $q\ge m$}.
\]
\end{enumerate}
\end{lemma}

\begin{proof}
Fix a neighborhood base $\{U_n\}_{n\in\w}$ at the point $z$ consisting of clopen subsets of $X$ such that $U_{n+1}\subseteq U_n$ for every $n\in\w$.

We shall inductively construct  increasing number sequences $(n_k)_{k\in\w}$, $(m_k)_{k\in\w}$ and a sequence $(S_k)_{k\in\w}$ of compact subsets of $X\setminus\{z\}$ such that for every $k\in\NN$ the following inductive conditions are satisfied:
\begin{itemize}
\item[$(a_k)$] $S_{k-1}\cap U_{m_{k}}=\emptyset$;
\item[$(b_k)$]
 $\sup_{n\ge n_{k}}\min\big\{|\mu_n(U_{m_{k}})-\mu(U_{m_{k}})|,\sup_{x\in X}\mu_n(\{x\})\big\}<\frac{1}{6^{k}}\mu(\{z\})\le\tfrac{1}{6^{k}}$;
\item[$(c_k)$] $S_{k}\subseteq U_{m_{k}}$ and $\mu_{n_{k}}(S_{k})>\mu_{n_{k}}(U_{m_{k}}\setminus\{z\})-\frac{1}{6^{k}}\mu(\{z\})$.
\end{itemize}

We start the inductive construction letting $m_0=0$ and $S_0=\emptyset$ and choosing $n_0\in\w$ such that the condition $(b_0)$ is satisfied. Assume that for some $k\in\NN$, we have constructed  a compact subset $S_{k-1}$ of $X\setminus\{z\}$ and numbers $n_{k-1},m_{k-1}$. As $S_{k-1}$ is a compact subset of $X\SM\{z\}$, there is a number $m_k>m_{k-1}$ such that $U_{m_k}\cap S_{k-1}=\emptyset$. Since $\lim_{n\in\w}\mu_n=\mu$ and $\lim_{n\in\w}\sup_{x\in X}\mu_n(\{x\})=0$, there exists a number $n_{k}>n_{k-1}$ satisfying the inductive condition $(b_k)$. Using the regularity of the measure $\mu_{n_{k}}$, choose a compact set $S_{k}\subseteq U_{m_{k}}\setminus\{z\}$ satisfying the condition $(c_k)$.  This completes the inductive step.

Observe that for every $k\in\NN$, the inductive conditions $(b_k)$ and $(c_k)$ imply
\begin{equation}\label{eq:new1}
\begin{aligned}
\mu_{n_{k}}(S_k)&>\mu_{n_k}(U_{m_{k}}\setminus\{z\})-\tfrac{1}{6^{k}}\mu(\{z\})=\mu_{n_k}(U_{m_{k}})-\mu_{n_k}(\{z\})-\tfrac{1}{6^{k}}\mu(\{z\})\\
&>\big(\mu(U_{m_k})-\tfrac{1}{6^{k}}\mu(\{z\})\big)-\tfrac{1}{6^{k}}\mu(\{z\})-\tfrac{1}{6^{k}}\mu(\{z\})\ge \big(1-\tfrac{3}{6^{k}}\big)\mu(\{z\}).
\end{aligned}
\end{equation}

\begin{claim}\label{cl1}
For every $k\in\IN$, the compact set $S_{k}$ can be represented as a union $\bigcup_{t\in 2^k}S_{k,t}$ of $2^k$ pairwise disjoint  compact sets such that
\begin{equation} \label{equ:pl-JNP-3}
\max_{t\in 2^k}\big|\mu_{n_k}(S_{n,t})-\tfrac{1}{2^k}\cdot\mu_{n_k}(S_k)\big|< \tfrac{4}{6^k}.
\end{equation}
\end{claim}

\begin{proof}
Identify $2^k$ with the finite ordinal $|2^k|=\{0,\dots,|2^k|-1\}$ via the bijective map $2^k\to \{0,\dots,|2^k|-1\}$, $t\mapsto \sum_{i=0}^{k-1} t(i)2^i$.
Since the compact space $S_k$ is zero-dimensional, it is homeomorphic to a subspace of the real line. This allows us to identify $S_k$ with a compact subset of the real line. For a real number $r$, let ${\downarrow}r=S_k\cap (-\infty,r]\subseteq S_k\subseteq\IR$. For every $t\in 2^k$, let
\[
s_t=\inf\{s\in S_k:\mu_{n_k}({\downarrow}s)\ge \tfrac{t+1}{2^k}\mu_{n_k}(S_k)\}.
\]
The regularity of the measure $\mu_{n_k}$ ensures that
\[
\begin{aligned}
\frac{t+1}{2^k}\mu_{n_k}(S_k) & \le\mu_{n_k}({\downarrow}s_t)\le \frac{t+1}{2^k}\mu_{n_k}(S_k)+\mu_{n_k}(\{s_t\})\stackrel{(b_k)}{<}
\frac{t+1}{2^k}\mu_{n_k}(S_k)+\frac1{6^{k}}\mu(\{z\})\\
&\stackrel{(\ref{eq:new1})}{<}\frac {t+1}{2^k}\mu_{n_k}(S_k)+\frac{\mu_{n_k}(S_{k})}{6^{k}(1-3/6^{k})}<\frac{t+2}{2^k}\mu_{n_k}(S_k),
\end{aligned}
\]
which implies $s_{t}<s_{t+1}$ if $t<2^k$. Let $r_{-1},r_{2^k-1}\in\IR\setminus S_k$ be any real numbers such that ${\downarrow}r_{-1}=\emptyset$ and ${\downarrow}r_{2^k-1}=S_k$. Since the space $S_k\subseteq\IR$ is zero-dimensional, for every $t\in 2^k-1$ we can choose a real number $r_t\in \IR\setminus S_k$ such that $s_t< r_t<s_{t+1}$ and $\mu_{n_k}({\downarrow}r_t)<\mu_{n_k}({\downarrow}s_t)+\tfrac{1}{6^k}$. Then
\[
|\mu_{n_k}({\downarrow}r_t)-\tfrac{t+1}{2^k}\mu_{n_k}(S_k)|\le |\mu_{n_k}({\downarrow}r_t)-\mu_{n_k}({\downarrow}s_t)|+|\mu_{n_k}({\downarrow}s_t)-\tfrac{t+1}{2^k}\mu_{n_k}(S_k)|<
\tfrac{1}{6^k}+\tfrac{1}{6^k}=\tfrac{2}{6^k}.
\]

It follows from $r_t\in\IR\setminus S_k$ that for every $t\in 2^k$ the subspace $S_{k,t}=S_k\cap({\downarrow}r_{t})\setminus ({\downarrow}r_{t-1})$ is compact and $S_k=\bigcup_{t\in 2^k}S_{k,t}$. It is clear that the family $(S_{k,t})_{t\in 2^k}$ consists of pairwise disjoint sets.

Finally, observe that
\[
\begin{aligned}
|\mu_{n_k}(S_{k,t}) &-\tfrac1{2^k}\mu_{n_k}(S_{k})|  = |\mu_{n_k}({\downarrow}r_t)-\mu_{n_k}({\downarrow}r_{t-1})-\tfrac{t+1}{2^k}\mu_{n_k}(S_k)+\tfrac{t}{2^k}\mu_{n_k}(S_k)|\\
& \leq|\mu_{n_k}({\downarrow}r_t)-\tfrac{t+1}{2^k}\mu_{n_k}(S_k)|+|\mu_{n_k}({\downarrow}r_{t-1})-\tfrac{t}{2^k}\mu_{n_k}(S_k)|
 <\tfrac2{6^k}+\tfrac{2}{6^k}=\tfrac{4}{6^k}.
\end{aligned}
\]
 \end{proof}
Now we continue the proof of the lemma. For every $s\in 2^{<\w}$, consider the set
\[
T_s=\bigcup_{k\ge |s|} \; \bigcup\{S_{k,t}:t\in 2^k,\;t{\restriction}_{|s|}=s\},
\]
where $S_{0,0}=S_0=\emptyset$.
Taking into account that the subsets $S_{k,t}$ of $U_{n_{k-1}}$ are compact and the sequence $\{U_k\}_{k\in\w}$ is a neighborhood base at $z$, we conclude that the  set $T_s$ is closed in $X\setminus\{z\}$.
The definition of the sets $T_s$ implies $T_{s\hat{\;}0}\cup T_{s\hat{\;}1}\subseteq T_s$. To show that $T_{s\hat{\;}0}\cap T_{s\hat{\;}1}=\emptyset$ for any $s\in 2^{<\w}$, it suffices to note that  the sequence $(S_k)_{k\in\w}$ is disjoint by the inductive conditions $(a_k)$ and $(c_k)$ and, for every $k\in\w$,  the family $\{S_{k,t}\}_{t\in 2^k}$ is disjoint by the construction.
\smallskip

Now, by induction on the binary tree $2^{<\w}$,  we shall find a family $(X_s)_{s\in2^{<\w}}$ of clopen subsets of $X\setminus\{z\}$ such that  $X_\emptyset=X\setminus\{z\}$ (the base of induction) and
\[
T_s\subseteq X_s,\;\;X_{s\hat{\;}0}\cup X_{s\hat{\;}1}=X_s,\;\;X_{s\hat{\;}0}\cap X_{s\hat{\;}1}=\emptyset \quad \mbox{ for every $s\in 2^{<\w}$}.
\]

Assume that for some $s\in 2^{<\w}$ the set $X_s$  has been constructed. Since $X$ is a zero-dimensional compact metrizable space, the space $X_s$ is zero-dimensional and being Lindel\"{o}f it is  strongly zero-dimensional in the sense that any disjoint closed sets in $X_s$ can be separated by clopen neighborhoods, see \cite[Theorem~6.2.7]{Eng}.  Since $T_{s\hat{\;}0}$ and $T_{s\hat{\;}1}$ are two closed disjoint sets in $T_s\subseteq X_s$, there are clopen sets $X_{s\hat{\;}0}$ and $X_{s\hat{\;}1}$ such that $T_{s\hat{\;}0}\subseteq X_{s\hat{\;}0}$, $T_{s\hat{\;}1}\subseteq X_{s\hat{\;}1}$, $X_{s\hat{\;}0}\cup X_{s\hat{\;}1}=X_s$, and  $X_{s\hat{\;}0}\cap X_{s\hat{\;}1}=\emptyset$. This completes the inductive step.
\smallskip

By the inductive construction, the family $(X_s)_{s\in 2^{<\w}}$ satisfies the condition (i) of the lemma. To verify the condition (ii), fix $s\in 2^{<\w}$, $\e>0$ and a clopen neighborhood $U$ of $z$ in $X$. Find $q\ge|s|$ such that
\begin{equation} \label{equ:pl-JNP-4}
U_{m_{k-1}}\subseteq U,\quad \tfrac{3}{6^k}+\tfrac{4}{3^k}<\tfrac{\e}{2}\mbox{ \ and \ }|\mu_{n_k}(U)-\mu(U)|<\tfrac{\e}{2} \quad \mbox{ for every $k\ge q$}.
\end{equation}
We claim that the number $m$ has the  property required in (ii). Indeed, if $t\in {\uparrow}s$, then $S_{k,t}\subseteq T_s \subseteq X_s$, and if $t\not\in {\uparrow}s$, then the disjointness of $T_{t{\restriction}s}$ and $T_s$ and the construction of $X_s$ imply $S_{k,t} \cap X_s=\emptyset$ and therefore
\begin{equation} \label{equ:pl-JNP-41}
S_k \cap X_s =\bigcup_{t\in 2^k} (S_{k,t}\cap X_s) =\bigcup_{t\in 2^k\cap {\uparrow}s}(S_{k,t}\cap X_s)=\bigcup_{t\in 2^k\cap {\uparrow}s}S_{k,t}.
\end{equation}

Note also that by $(c_k)$ and (\ref{equ:pl-JNP-4}), $S_k \subseteq U_{m_k}\subseteq U$. Now, for any $k\ge q$, we have
\[
\begin{aligned}
\mu_{n_k}(U\cap X_s) & \le\mu_{n_k}(U\setminus S_k)+\mu_{n_k}(S_k\cap X_s)\\
& \stackrel{(\ref{equ:pl-JNP-41})}{=} \mu_{n_k}(U\setminus S_k)+\sum_{t\in 2^k\cap{\uparrow}s}\mu_{n_k}(S_{k,t}) \\
&=\mu_{n_k}(U\setminus S_k)+\mu_{n_k}(S_k)-\sum_{t\in 2^k \,\setminus\,{\uparrow}s}\mu_{n_k}(S_{k,t}) \\
& \stackrel{(\ref{equ:pl-JNP-3})}{<} \mu_{n_k}(U)-(2^k-2^{k-|s|})\big(\tfrac{1}{2^k}\mu_{n_k}(S_k)-\tfrac{4}{6^{k}}\big)\\
&\stackrel{(\ref{eq:new1})}{<} \mu_{n_k}(U)-\big(1-\tfrac{1}{2^{|s|}}\big)\big(\mu(U_{m_{k}})-\tfrac{3}{6^k}-\tfrac{4}{3^k}\big)\\
& \stackrel{(\ref{equ:pl-JNP-4})}{<}
\big(\mu(U)+\tfrac{\e}{2}\big)-\mu(U_{m_k})+\tfrac{1}{2^{|s|}}+\tfrac{3}{6^k}+\tfrac{4}{3^{k}} \\
&\stackrel{(\ref{equ:pl-JNP-4})}{<}\mu(U\setminus U_{m_k})+\e+\tfrac{1}{2^{|s|}}\le \mu(U\setminus\{z\})+\tfrac{1}{2^{|s|}}+\e,
\end{aligned}
\]
as desired.
\end{proof}


\section{An implication of $\diamondsuit$} \label{sec:enumeration}


In this section we apply Jensen's diamond principle $\diamondsuit$ to producing a special enumeration of the set $(P(2^{\w_1}))^\w$ by countable ordinals.

Jensen's diamond principle is the following statement (see \cite{Jensen}):
\begin{itemize}
\item[$(\diamondsuit)$] {\em there exists a transfinite sequence $(x_\alpha)_{\alpha\in\w_1}\in\prod_{\alpha\in\w_1}2^\alpha$ such that for every $x\in 2^{\omega_1}$, the set $\{\alpha\in\w_1:x{\restriction}_\alpha=x_\alpha\}$ is stationary in $\w_1$.}
\end{itemize}
A set $S\subseteq\w_1$ is {\em stationary} if $S$ has nonempty intersection with any closed unbounded subset of $\w_1$. By \cite{Jensen}, $\diamondsuit$ implies CH and follows from the G\"odel Constructibility Axiom $V=L$ (which is consistent with ZFC).

We shall apply $\diamondsuit$ to produce a nice enumerations of elements of the limit spaces of  subcontinuous $\w_1$-spectra.

Let $\lambda$ be any ordinal. A {\em $\lambda$-spectrum} if an indexed family $\Sigma=(X_\alpha,\pi_\alpha^\beta:\alpha\le \beta<\lambda)$ consisting of sets $X_\alpha$ and functions $\pi_\alpha^\beta:X_\beta\to X_\alpha$  such that $\pi_\alpha^\gamma=\pi^\gamma_\beta\circ\pi^\beta_\alpha$ for any ordinals $\alpha\le\beta\le\gamma$ in $\lambda$. 
The spectrum $\Sigma$ is defined to be {\em subcontinuous} if for every limit ordinal $\beta<\lambda$ and any distinct elements $x,x'\in X_\beta$ there exists an ordinal $\alpha<\beta$ such that $\pi^\beta_\alpha(x)\ne\pi^\beta_\alpha(x')$.

\begin{proposition}\label{p:diamond}
Let $\Sigma=(X_\alpha,\pi^\beta_\alpha:\alpha\le\beta\le\w_1)$ be a subcontinuous $(\w_1+1)$-spectrum such that for every $\alpha<\w_1$ the set $X_\alpha$ has cardinality $|X_\alpha|\le\mathfrak c$. Under $\diamondsuit$ there exists a transfinite sequence $(x_\alpha)_{\alpha\in\w_1}\in \prod_{\alpha\in\w_1}X_\alpha$ such that for every $x\in X_{\w_1}$, the set $\{\alpha\in\w_1:x_\alpha= \pi^{\w_1}_\alpha(x)\}$ is stationary in $\w_1$.
\end{proposition}

\begin{proof} Assuming $\diamondsuit$, we obtain a transfinite sequence $(y_\alpha)_{\alpha\in\w_1}\in\prod_{\alpha\in\w_1}2^\alpha$ such that for every $y\in 2^{\omega_1}$, the set $\{\alpha\in\w_1:y_\alpha=y{\restriction}_\alpha\}$ is stationary in $\w_1$.

Identifying $\w_1$ with the ordinal $\w\cdot\w_1$, we conclude that $\diamondsuit$ implies the existence of a transfinite sequence $(z_\alpha)_{\alpha\in\w_1}\in\prod_{\alpha\in \w_1}(2^\w)^\alpha$ such that for any $z\in (2^\w)^{\w_1}$ the set $\{\alpha\in\w_1:z_\alpha=z{\restriction}_\alpha\}$ is stationary in $\w_1$.

For every $\alpha\in\w_1$, the set $X_\alpha$ has cardinality $|X_\alpha|\le|2^\w|$ and hence admits an injective function $f_\alpha:X_\alpha\to 2^\w$. For every $\beta\in\w_1$, consider the  function
$$g_{\beta}:X_\beta\to (2^\w)^{\beta},\quad g_{\beta}:x\mapsto (f_\alpha\circ\pi^\beta_\alpha)_{\alpha\in\beta}.$$

\begin{claim}
For every limit ordinal $\beta\in\w_1$ the function $g_\beta:X_\beta\to (2^\w)^\beta$ is injective.
\end{claim}

\begin{proof}
Given any distinct elements $x,x'\in X_\beta$, use the subcontinuity of the spectrum $\Sigma$ and find an ordinal $\alpha<\beta$ such that $\pi^\beta_\alpha(x)\ne \pi^\beta_\alpha(x')$. The injectivity of the function $f_\alpha:X_\alpha\to 2^\w$ implies that $f_\alpha(\pi^{\w_1}_\alpha(x))\ne f_\alpha(\pi^{\w_1}_\alpha(x'))$. Let $\pr^\beta_\alpha:(2^\w)^\beta\to 2^\w$, $\pr^\beta_\alpha:x\mapsto x(\alpha)$, be the projection of $(2^\w)^\beta$ onto the $\alpha$-th coordinate. It follows from
$$\pr^\beta_\alpha\circ g_\beta(x)=f_\alpha\circ\pi^\beta_\alpha(x)\ne f_\alpha\circ\pi^\beta_\alpha(x')=\pr^\beta_\alpha\circ g_\beta(x')$$ that $g_\beta(x)\ne g_\beta(x')$, which means that the function $g_\beta:X_\beta\to (2^\w)^\beta$ is injective.
\end{proof}

Let $(x_\alpha)_{\alpha\in\w_1}\in\prod_{\alpha\in\w_1}X_\alpha$ be any sequence such that for every $\alpha\in\w_1$ with $z_\alpha\in g_\alpha[X_\alpha]$ we have $g_\alpha(x_\alpha)=z_\alpha$. We claim that the transfinite sequence $(x_\alpha)_{\alpha\in\w_1}$ has the required property.

Fix any $x\in X_{\w_1}$ and consider the element $z=g_{\w_1}(x)\in (2^\w)^{\w_1}$. The choice of the transfinite sequence $(z_\alpha)_{\alpha\in\w_1}$ ensures that the set $S=\{\alpha\in\w_1:z{\restriction}_\alpha=z_\alpha\}$ is stationary in $\w_1$.

To see that the set $\{\alpha\in\w_1:\pi^{\w_1}_\alpha(x)=x_\alpha\}$ is stationary, it suffices to show that it contains the stationary set  $S\cap\w'_1$, where $\w_1'$ is the closed unbounded set of all limit countable ordinals.

Given any limit ordinal $\beta\in S$, observe that $z_\beta=z{\restriction}_\beta=g_{\w_1}(x){\restriction}_\beta=(f_\alpha\circ\pi^{\w_1}_\alpha(x))_{\alpha\in\beta}=g_\beta(\pi^{\w_1}_\beta(x))\in g_\beta[X_\beta]$ and hence $z_\beta=g_\beta(x_\beta)$. On the other hand, $$z_\beta=z{\restriction}_\beta=g_{\w_1}(x){\restriction}_\beta=g_\beta(\pi^{\w_1}_\beta(x)).$$Since $g_\beta(x_\beta)=z_\beta=g_\beta(\pi^{\w_1}_\beta(x))$, the injectivity of $g_\beta$ implies $x_\beta=\pi^{\w_1}_\beta(x)$. Therefore, $S\cap\w_1'\subseteq\{\alpha\in\w_1:\pi^{\w_1}_\alpha(x)=x_\alpha\}$ and the set $\{\alpha\in\w_1:\pi^{\w_1}_\alpha(x)=x_\alpha\}$   is stationary.
\end{proof}

Now we apply Proposition~\ref{p:diamond} to subcontinuous $\lambda$-spectra in the category $\Comp$ of compact Hausdorff spaces and their continuous maps. A {\em $\lambda$-spectrum} in the category $\Comp$ is a $\lambda$-spectrum $(X_\alpha,\pi^\beta_\alpha:\alpha\le\beta<\lambda)$ consisting of compact Hausdorff spaces $X_\alpha$ and continuous bounding maps $\pi^\beta_\alpha$. An example of such spectrum is the $(\w+1)$-spectrum $(2^\alpha,\pi^\beta_\alpha:\alpha\le\beta\le\w_1)$, where the cubes $2^\alpha$ are endowed with the Tychonoff product topology and the projection maps $\pi_\alpha^\beta:2^\beta\to 2^\alpha$, $\pi_\alpha^\beta:x\mapsto x{\restriction}_\alpha$, are continuous.

A functor $F:\Comp\to\Comp$ in the category $\Comp$ is defined to be {\em subcontinuous} if for any ordinal $\lambda$ and any subcontinuous $\lambda$-spectrum $(X_\alpha,\pi^\beta_\alpha:\alpha\le\beta<\lambda)$ in the category $\Comp$,  the spectrum $(FX_\alpha,F\pi^\beta_\alpha:\alpha\le\beta<\lambda)$ is subcontinuous. By \cite[2.3.3]{TZ}, the functor of probability measures $P:\Comp\to\Comp$ is subcontinuous and so is the functor $(\cdot)^\w:\Comp\to\Comp$ of taking the countable power. Then the composition $P^\w:\Comp\to\Comp$ of the functors $P$ and $(\cdot)^\w$ is a subcontinuous functor, too. This functor assigns to each compact Hausdorff space $X$ the countable power $P^\w(X)$ of the space $P(X)$. To any continuous function $f:X\to Y$ between compact Hausdorff space the function $P^\w$ assigns the continuous function $P^\w f:P^\w(X)\to P^\w(Y)$, $P^\w:(\mu_n)_{n\in\w}\mapsto (Pf(\mu_n))_{n\in\w}$.

\begin{corollary}\label{c:diamond}
Under $\diamondsuit$ there exists a transfinite sequence $(\mu_\alpha)_{\alpha\in\w_1}\in\prod_{\alpha\in\w_1}P^\w(2^\alpha)$ such that for every $\mu\in P^\w(2^{\w_1})$, the set $\{\alpha\in\w_1:\mu_\alpha= P^\w\pi^{\w_1}_\alpha(\mu)\}$ is stationary in $\w_1$.
\end{corollary}

\begin{proof} By the subcontinuity of the functor $P^\w$, the $(\w_1+1)$-spectrum $(P^\w(2^\alpha),P^\w\pi^\beta_\alpha:\alpha\le\beta\le\w_1)$ is subcontinuous. For every countable ordinal $\alpha$, the compact space $P^\w(2^\w)$ is metrizable and hence has cardinality $\le\mathfrak c$. Applying Proposition~\ref{p:diamond}, we obtain a transfinite sequence $(\mu_\alpha)_{\alpha\in\w_1}\in\prod_{\alpha\in\w_1}P^\w(2^\alpha)$ with the desired property.
\end{proof}


\section{Proof of Theorem~\ref{t:main}}\label{s:main}


Consider the $(\w_1+1)$-spectrum $(2^\alpha,\pi^\beta_\alpha:\alpha\le\beta\le\w_1\}$ consisting of the Cantor cubes $2^\alpha$ and projections $\pi^\beta_\alpha:2^\beta\to 2^\alpha$, $\pi^\beta_\alpha:x\mapsto x{\restriction}\alpha$.

Assume $\diamondsuit$. By Corollary~\ref{c:diamond}, there exists a transfinite sequence $$\big((\mu^\alpha_n)_{n\in\w}\big)_{\alpha\in\w_1}\in\prod_{\alpha\in\w_1}\big(P(2^{\alpha})\big)^\w$$ such that for any sequence of measures $(\mu_n)_{n\in\w}\in \big(P(2^{\w_1})\big)^\w$ the set
\[
\{\alpha\in\w_1:\;\mu_n^\alpha=P\pi_\alpha^{\w_1}(\mu_n)\; \mbox{ for every } n\in\w\}
\]
is stationary in $\w_1$.


Let $\Lambda$ be the set of all nonzero limit ordinals $\alpha\in\w_1$ such that $(\mu^\alpha_n)_{n\in\w}$ is a sequence of pairwise distinct Dirac measures that converges to the Dirac measure $\mu^\alpha_0\in X_\alpha\subseteq P(X_\alpha)$.

Let $\Omega$ be the set of all nonzero limit ordinals $\alpha\in\w_1$ such that
 $\lim_{n\in\w}\sup_{x\in 2^\alpha}\mu^\alpha_n(\{x\})=0$.

Observe that the sets $\Lambda$ and $\Omega$ are disjoint.
\smallskip

Now we shall inductively construct families of sets $(I_{\alpha})_{\alpha\in\w_1}$, $(X_{\alpha})_{\alpha\in\w_1}$, $(A_{\alpha})_{\alpha\in\w_1}$, $(B_{\alpha})_{\alpha\in\w_1}$ such that  for every $\alpha\in\w_1$, the following conditions are satisfied:
\begin{enumerate}\itemsep=2pt
\item[(1)] if $\alpha\le\w$, then $X_\alpha=A_\alpha=2^\alpha$, $B_\alpha=\emptyset$ and $I_\alpha=\w$;
\item[(2)] if the ordinal $\alpha$ is limit, then $X_\alpha=\bigcap_{\gamma<\alpha}(\pi_\gamma^\alpha)^{-1}[X_\gamma]$;
\item[(3)] $I_\alpha$ is an infinite subset of $\w$;
\item[(4)] $X_\alpha,A_\alpha,B_\alpha$ are  closed crowded subspaces of $2^\alpha$ and $X_\alpha\ne\emptyset$;
\item[(5)] $X_\alpha=A_\alpha\cup B_\alpha$ and  $|A_\alpha\cap B_\alpha|\le 1$;
\item[(6)] $X_{\alpha+1}=(A_\alpha\times\{0\})\cup(B_\alpha\times\{1\})\subseteq X_\alpha\times 2\subseteq 2^{\alpha+1}$;
\item[(7)] $\pi^\alpha_\gamma[X_\alpha]=X_\gamma$ for every $\gamma<\alpha$;
\item[(8)] if $\alpha\in\Lambda$ and  $\bigcup_{n\in\w}\supp(\mu^\alpha_n)\subseteq X_\alpha$, then\\
$
\bigcup_{n\in\w} \supp(\mu^\alpha_{2n+1})\subseteq  A_\alpha$ and $\bigcup_{n\in\w} \supp(\mu^\alpha_{2n+2})\subseteq  B_\alpha$;
\item[(9)] if $\alpha\in \Omega$ and $\bigcup_{n\in\w}\supp(\mu^\alpha_n)\subseteq X_\alpha$, then for any sequence of measures $(\lambda_n)_{n\in\w}\in\big(P(X_{\alpha+\w})\big)^\w$ with $P\pi^{\alpha+\w}_\alpha(\lambda_n)=\mu^\alpha_n$ for all $n\in\w$, every accumulation point of the sequence $(\lambda_n)_{n\in I_\alpha}$ in $P(X_{\alpha+\w})$ is a nonatomic measure on the space $X_{\alpha+\w}$. 
\end{enumerate}
\smallskip

To start the inductive construction, for every $\alpha\le\w$, put $X_\alpha=A_\alpha=2^\alpha$, $B_\alpha=\emptyset$ and $I_\alpha=\w$. Assume that for some limit ordinal $\alpha$ and all $\gamma<\alpha$ we have constructed nonempty closed subspaces $X_\gamma\subseteq 2^\gamma$ that have no isolated points and satisfy the inductive condition (7). Put $X_\alpha=\bigcap_{\gamma<\alpha}(\pi_\gamma^\alpha)^{-1}[X_\gamma]$ and observe that $X_\alpha$ is crowded and satisfies the inductive condition (7).

To define the other sets we consider the following three cases.
\smallskip

\noindent {\em Case 1: Assume that $\alpha\notin \Lambda\cup\Omega$ or $\bigcup_{n\in\w}\supp(\mu_n^\alpha)\not\subseteq X_\alpha$}. In this case put
\[
I_{\alpha+n}=\w,\; X_{\alpha+n}=A_{\alpha+n}=X_\alpha\times\{0\}^n\; \mbox{  and } \; B_{\alpha+n}=\emptyset
\]
for every $n\in\w$.
It is easy to see that the families $(I_\beta)_{\beta<\alpha+\w}$,  $(X_\beta)_{\beta<\alpha+\w}$,  $(A_\beta)_{\beta<\alpha+\w}$,  $(B_\beta)_{\beta<\alpha+\w}$ satisfy the inductive conditions (1)--(9). 
\smallskip


\noindent {\em Case 2: Assume that $\alpha\in\Lambda$ and $\bigcup_{n\in\w}\supp(\mu^\alpha_n)\subseteq X_\alpha$.}  In this case the definition of the set $\Lambda$ ensures that the sequence $(\mu^\alpha_n)_{n\in\w}$ consists of pairwise distinct Dirac measures that converge to the Dirac measure $\mu^\alpha_0$ in $2^\alpha$. Since the sets $\bigcup_{n\in\w}\supp(\mu^\alpha_{2n+1})$ and  $\bigcup_{n\in\w}\supp(\mu^\alpha_{2n+2})$ are closed and disjoint in the zero-dimensional metric space $X_\alpha\setminus \supp(\mu^\alpha_0)$, we can use \cite[Theorem~6.2.7]{Eng} (as in the proof of Lemma \ref{l:atom}) and find two disjoint clopen sets $V_\alpha,W_\alpha$ in $X_\alpha\setminus\supp(\mu^\alpha_0)$ such that
\[
\bigcup_{n\in\w}\supp(\mu^\alpha_{2n+1})\subseteq V_\alpha,\;\bigcup_{n\in\w}\supp(\mu^\alpha_{2n+2})\subseteq W_\alpha,\;V_\alpha\cap W_\alpha=\emptyset,\;V_\alpha\cup W_\alpha=X_\alpha\setminus \supp(\mu^\alpha_0).
\]
Since the space $X_\alpha\setminus \supp(\mu^\alpha_0)$ is crowded, so are its clopen subspaces $V_\alpha$ and $W_\alpha$.

Consider the subsets $A_\alpha:=V_\alpha\cup \supp(\mu^\alpha_0)$, $B_\alpha:=W_\alpha\cup \supp(\mu^\alpha_0)$ of $2^\alpha$, and observe that they are closed in $2^\alpha$ and have no isolated points. Define
\[
X_{\alpha+1}=(A_\alpha\times\{0\})\cup(B_\alpha\times\{1\})\subseteq X_\alpha\times 2.
\]
Set $I_\alpha:=\w$. For every ordinal $n\in[1,\w)$,  put
\[
I_{\alpha+n}:=\w,\;\;X_{\alpha+n}=A_{\alpha+n}:=X_{\alpha+1}\times \{0\}^{n-1}\subseteq X_\alpha\times 2^n \; \mbox{ and } \; B_{\beta+n}:=\emptyset.
\]
It is easy to see that the  families $(I_\beta)_{\beta<\alpha+\w}$, $(X_\beta)_{\beta<\alpha+\w}$, $(A_\beta)_{\beta<\alpha+\w}$ and $(B_\beta)_{\beta<\alpha+\w}$ satisfy the inductive conditions (1)--(9).
\smallskip

\noindent{\em Case 3: Assume that $\alpha\in\Omega$ and $\bigcup_{n\in\w}\supp(\mu^\alpha_0)\subseteq X_\alpha$.}  In this case the definition of the set $\Omega$ ensures that $\lim_{n\in\w}\sup_{x\in 2^\alpha}\mu^\alpha_n(x)=0$. Since $P(X_\alpha)$ is a compact metrizable space, there exists an infinite set $J_\alpha\subseteq\w$ such that the sequence $(\mu^\alpha_n)_{n\in J_\alpha}$ converges to some measure $\mu\in P(X_\alpha)$. If the measure $\mu$ is nonatomic, then, for every ordinal $n<\w$, we put
\[
I_{\alpha+n}=J_\alpha,\;X_{\alpha+n}=A_{\alpha+n}=X_\alpha\times\{0\}^n \; \mbox{ and }\; B_{\alpha+n}=\emptyset, 
\]
and observe that the families $(I_\beta)_{\beta<\alpha+\w}$, $(X_\beta)_{\beta<\alpha+\w}$, $(A_\beta)_{\beta<\alpha+\w}$ and $(B_\beta)_{\beta<\alpha+\w}$ satisfy the inductive conditions (1)--(9).

\smallskip

Now we assume that the measure $\mu$ has atoms. Let $T_\mu=\{x\in X_\alpha:\mu(\{x\})>0\}$ be the set of all atoms of the measure $\mu$. The additivity of the probability measure $\mu$ ensures that the set $T_\mu$ is at most countable. Since the compact metrizable space $X_\alpha$ is crowded, the $G_\delta$-subset $G_\alpha:=X_\alpha\setminus T_\mu$ of $X_\alpha$ is dense in $X_\alpha$ (by the Baire Theorem) and also crowded.

The set $T_\mu$ is at most countable and hence admits a well-order $\preceq$ such that for every $a\in T_\mu$ the set ${\downarrow}a=\{x\in T_\mu:x\preceq a\}$ is finite.

Using Lemma~\ref{l:atom} inductively, for every $a\in T_\mu$ we can choose an infinite subset $J_{\alpha,a}$ of $J_\alpha$ and a family $(Z_{(a,s)})_{s\in 2^{<\w}}$ of clopen subsets of $X_\alpha\setminus\{a\}$ such that $Z_{(a,\emptyset)}=X_\alpha\setminus\{a\}$, $J_{\alpha,a}\subseteq J_{\alpha,b}$ for any $b\in T_\mu$ with $b\preceq a$ and for every $s\in 2^{<\w}$ the following conditions are satisfied:
\begin{itemize}
\item[$(\dag)$] $Z_{(a,s\hat{\;}0)}\cup Z_{(a,s\hat{\;}1)}=Z_{(a,s)}$ and $Z_{(a,s\hat{\;}0)}\cap Z_{(a,s\hat{\;}1)}=\emptyset$;
\item[$(\ddag)$] for any clopen neighborhood $U\subseteq X_\alpha$ of $a$ and each $\e>0$ there exists an $m\in\w$ such that for every $n\in J_{\alpha,a}$ with $n\ge m$ we have $\mu^\alpha_n(U\cap Z_{(a,s)})<\mu(U\SM\{a\})+\frac1{2^{|s|}}+\e$.
\end{itemize}

Choose any infinite set $I_\alpha\subseteq J_\alpha$ such that $I_\alpha\subseteq^* J_{\alpha,a}$ for every $a\in T_\mu$.

Fix a bijective function $\xi:\w\to T_\mu\times 2^{<\w}$ such that for any $a\in T_\mu$ and $s<t$ in $2^{<\w}$ we have $\xi^{-1}(a,s)<\xi^{-1}(a,t)$ (to construct such bijection it is sufficient to consider a partition of $\w$ onto $|T_\mu|$ infinite sets, and on every set of the partition to consider a copy of the bijection  $2^{<\w}\to\w$, $t\mapsto \sum_{i=0}^{k-1} t(i) 2^i$).
For every $n\in\w$, let $\gamma_n:G_\alpha\to 2$ be the characteristic function of the clopen subset $Z_{\xi(n)}\cap G_\alpha$ in $G_\alpha$.

Now, for every $n\in\w$, we define the sets $ X_{\alpha+n}, A_{\alpha+n}$ and $B_{\alpha+n}$ as follows:
\begin{itemize}
\item[$\bullet$] $X_{\alpha+n}$ is the closure of the set $\{(z,\gamma_0(z),\dots,\gamma_{n-1}(z)):z\in G_\alpha\}$ in $X_\alpha\times 2^n$;
\item[$\bullet$]  $A_{\alpha+n}$ is the closure of the set $\{(z,\gamma_0(z),\dots,\gamma_{n-1}(z)):z\in G_\alpha\setminus Z_{\xi(n)}\}$ in $X_\alpha\times 2^n$;
\item[$\bullet$]  $B_{\alpha+n}$ is the closure of the set $\{(z,\gamma_0(z),\dots,\gamma_{n-1}(z)):z\in G_\alpha\cap Z_{\xi(n)}\}$ in $X_\alpha\times 2^n$.
\end{itemize}
Since $Z_{\xi(n)}\cap G_\alpha$ is a clopen set in  $G_\alpha$,  the spaces $X_{\alpha+n},A_{\alpha+n},B_{\alpha+n}$ are crowded, being closures of the the topological copies of the crowded spaces $G_\alpha$, $G_\alpha\setminus Z_{\xi(n)}$ and $G_\alpha\cap Z_{\xi(n)}$. This means that the condition (4) is satisfied.

Next, we show that the condition (5) is satisfied. For every $n\in\w$, the equality $X_{\alpha+n}=A_{\alpha+n}\cup B_{\alpha+n}$ holds by the definitions of the sets $X_{\alpha+n},A_{\alpha+n},B_{\alpha+n}$. To see that $|A_{\alpha+n}\cap B_{\alpha+n}|\le 1$, take any points $(x_0,t_0),(x_1,t_1)\in A_{\alpha+n}\cap B_{\alpha+n}\subseteq X_{\alpha+n}\subseteq X_\alpha\times 2^n$, where $x_0,x_1\in X_\alpha$ and $t_0,t_1\in 2^n$. Find $a\in T_\mu$ and $s\in 2^{<\w}$ such that $(a,s)=\xi(n)$. Taking into account that all $Z_{(a,s)}$ are clopen in $X_\alpha\SM\{ a\}$, the condition $(\dag)$ implies
\[
x_0,x_1\in\pi^{\alpha+n}_\alpha(A_{\alpha+n})\cap \pi^{\alpha+n}_\alpha(B_{\alpha+n})\subseteq \overline{G_\alpha\setminus Z_{\xi(n)}}\cap \overline{G_\alpha\cap Z_{\xi(n)}}\subseteq\{a\},
\]
where the closure is taken in the compact space $X_\alpha$.
Therefore  $x_0=a=x_1$. Assuming that $t_0\ne t_1$, we can find an ordinal $k\in n$ such that $t_0(k)\ne t_1(k)$. Let $\xi(k)=(b,t)$ for some $b\in T_\mu$ and $t\in 2^{<\w}$. If $b\ne a$, then either $Z_{\xi(k)}$ or $X_\alpha\setminus Z_{\xi(k)}$ is a neighborhood of $a$ in $X_\alpha$. By the definition of $\gamma_k$, in the first case we get $t_0(k)=t_1(k)=1$, and in the second case we get $t_0(k)=t_1(k)=0$. But this contradicts the inequality $t_0(k)\ne t_1(k)$, and hence $b=a$. By the choice of the function $\xi$, either $t\le s$ or $t$ and $s$ are incomparable in the tree $2^{<\w}$. If $t\le s$, then $Z_{\xi(n)}\subseteq Z_{\xi(k)}$ and hence $\gamma_k(G_\alpha\cap Z_{\xi(n)})=\{1\}$ and $t_0(k)=t_1(k)=1$. If $s$ and $t$ are incomparable, then $Z_{\xi(n)}\cap Z_{\xi(k)}=\emptyset$ and $\gamma_k(G_\alpha\cap Z_{\xi(n)})=\{0\}$, which implies $t_0(k)=t_1(k)=0$. In both cases we get a contradiction with the choice of $k$. Therefore, $|A_{\alpha+n}\cap B_{\alpha+n}|\le 1$ and the condition (5) holds.

The conditions (6) and (7) follow from the definition of the spaces $X_{\alpha+n+1}$, $A_{\alpha+n},B_{\alpha+n}$ and the inclusions $\gamma_n(G_\alpha\setminus Z_{\xi(n)})\subseteq\{0\}$ and $\gamma_n(G_\alpha\cap Z_{\xi(n)})\subseteq\{1\}$ holding for all $n\in\w$.
\smallskip



Finally, we check the inductive condition (9). Choose a sequence of measures $(\lambda_n)_{n\in \w}\subseteq (P(X_{\alpha+\w}))^\w$ such that $P\pi^{\alpha+\w}_\alpha(\lambda_n)=\mu^\alpha_n$ for all $n\in\w$. Let $\lambda\in P(X_{\alpha+\w})$ be an accumulation point of the sequence $(\lambda_n)_{n\in I_\alpha}$. The convergence of the sequence $(\mu^\alpha_n)_{n\in J_\alpha}$ to the measure $\mu$ and the continuity of the map $P\pi^{\alpha+\w}_\alpha:P(X_{\alpha+\w})\to P(X_\alpha)$ imply that $P\pi^{\alpha+\w}_\alpha(\lambda)=\mu$. We should prove that the measure $\lambda$ is nonatomic. To derive a contradiction, assume that $\lambda(\{b\})>0$ for some $b\in X_{\alpha+\w}$. Consider the point $a=\pi^{\alpha+\w}_\alpha(b)$ and conclude that $\mu(\{a\})=\lambda((\pi^{\alpha+\w}_\alpha)^{-1}(a))\ge\lambda(\{b\})>0$ and hence $a\in T_\mu$. By the regularity of the measure $\mu$, there exists a clopen neighborhood $U\subseteq X_\alpha$ of $a$ such that $\mu(U\setminus\{a\})<\frac{1}{4}\lambda(\{b\})$. Find $l\in\w$ such that $\frac{2}{2^l}<\frac{1}{4}\lambda(\{b\})$. By the property $(\ddag)$ and almost inclusion $I_\alpha\subseteq^* J_{\alpha,a}$, there exists an $m\in\w$ such that for every $s\in 2^l$ and every $n\in I_\alpha$ with $n\ge m$ we have
\begin{equation}\label{equ:pl-JNP-5}
\mu^\alpha_n(U\cap Z_{a,s})<\mu(U\setminus\{a\})+\tfrac{1}{2^{|s|}}+\tfrac{1}{2^l}<\tfrac{2}{4}\lambda(\{b\}).
\end{equation}
Since $\lim_{n\in\w}\sup_{x\in 2^\alpha}\mu^\alpha_n(\{x\})=0$, we can replace $m$ by a larger number, if necessary, and assume additionally that $\mu_n^\alpha(\{a\})<\frac{1}{4}\lambda(\{b\})$ for all $n\ge m$.

 Consider the continuous map
\[
\gamma:G_\alpha\to X_{\alpha+\w}\subseteq X_\alpha\times 2^\w, \quad \gamma:x\mapsto \big(x,(\gamma_n(x))_{n\in\w}\big),
\]
and observe that the image $\gamma[G_\alpha]$ is dense in $X_{\alpha+\w}$. For every $s\in 2^l$, consider the closed set $W_s=\overline{\gamma[G_\alpha\cap Z_{(a,s)}]}$ in $X_{\alpha+\w}$.
We claim that the family $\{W_s\}_{s\in 2^l}$ is disjoint. Indeed, suppose for a contradiction that for some $s_0,s_1\in 2^l$, the intersection $W_{s_0}\cap W_{s_1}$ contains a point, say $\xxx$. Since $X_{\alpha+\w}$ is metrizable, there are two sequences $\{z_n^0\}_{n\in\w}$ and $\{z_n^1\}_{n\in\w}$ in $G_\alpha\cap Z_{(a,s_0)}$ and $G_\alpha\cap Z_{(a,s_1)}$, respectively, such that $\lim_n\gamma(z_n^0)=\xxx=\lim_n\gamma(z_n^1)$. Set $m:=\xi^{-1}(a,s_0)$. Since, by  $(\dag)$, $Z_{(a,s_0)}\cap Z_{(a,s_1)}=\emptyset$, the definition of $\gamma_m$ implies that $\gamma_m(z_n^0)=1$   and $\gamma_m(z_n^1)=0$ for every $n\in\w$, and hence $\lim_n\gamma(z_n^0)\not=\lim_n\gamma(z_n^1)$. This contradiction proves the claim.
Then the density of  $\gamma[G_\alpha]$ implies that the finite family $\{W_s\}_{s\in 2^l}$ covers the compact space $X_{\alpha+\w}$. This fact and the claim imply that $\{W_s\}_{s\in 2^l}$ is a disjoint clopen cover of $X_{\alpha+\w}$.

Find a  unique $s\in 2^l$ such that the clopen set $W_s:=\overline{\gamma(G_\alpha\cap Z_{(a,s)})}$ contains the point $b$. Then the set $W:=W_s\cap (\pi^{\alpha+\w}_\alpha)^{-1}(U)$ is a clopen neighborhood of the point $b$ in $X_{\alpha+\w}$.
Since $\lambda$ is an accumulating point of the sequence $(\lambda_n)_{n\in I_\alpha}$, there exists $n\in I_\alpha$ with $n\ge m$ such that $\lambda_n(W)>\frac{3}{4}\lambda(\{b\})$.

Taking into account that $Z_{(a,s)}$ is a closed subset of $X_\alpha\setminus \{a\}$, we conclude that
\[
W_s\subseteq (\pi_\alpha^{\alpha+\w})^{-1}\big(\{a\}\cup Z_{(a,s)}\big)\; \mbox{ and } \; W\subseteq  (\pi_\alpha^{\alpha+\w})^{-1}\big(\{a\}\cup (U\cap Z_{(a,s)})\big).
\]
Now observe that
\begin{multline*}
\mu^\alpha_n(U\cap Z_{(a,s)})=\mu^\alpha_n((U\cap Z_{(a,s)})\cup\{a\})-\mu_n^\alpha(\{a\})\\
>\lambda_n\big((\pi_\alpha^{\alpha+\w})^{-1}((U\cap Z_{(a,s)})\cup\{a\})\big)-\tfrac14\lambda(\{b\})\ge
 \lambda_n(W)-\tfrac14\lambda(\{b\})>\tfrac{2}{4}\lambda(\{b\}),
\end{multline*}
which contradicts (\ref{equ:pl-JNP-5}). This finishes the inductive construction.
\smallskip

After completing the inductive construction, consider the closed subspace
\[
K=\bigcap_{\alpha\in\w_1}\{x\in 2^{\w_1}:x{\restriction}_\alpha\in X_\alpha\}
 \]
of the compact space $2^{\w_1}$. Below we show that the compact space $K$ satisfies all conditions of Theorem~\ref{t:main}. 


\begin{claim}\label{cl:simple}
The compact space $K$ is simple and crowded.
\end{claim}

\begin{proof}
The space $K$ is simple by  (4)--(7). To show that $K$ is crowded, fix a point $z\in K$ and an arbitrary basic neighborhood $U=\{x\in K:x{\restriction}_F=z{\restriction}_F\}$ of $z$, where $F\subseteq \w_1$. Choose a countable ordinal $\alpha\in\w_1$ such that $F\subseteq \alpha$. The space $X_\alpha$ is crowded (by (4)) and hence contains a point $x_\alpha\in X_\alpha$ such that $x_\alpha\in \pi^{\w_1}_\alpha(U)\setminus \{\pi^{\w_1}_\alpha(z)\}$. Then any $x\in (\pi^{\w_1}_\alpha)^{-1}(x_\alpha) \cap K$ belongs to $U\setminus\{z\}$, witnessing that $K$ is crowded.
\end{proof}

\begin{claim}\label{cl:seq}
The space $K$ is sequentially discrete.
\end{claim}

\begin{proof}
To derive a contradiction, assume that $K$ contains a nontrivial convergent sequence. Then there exists a sequence $\{\mu_n\}_{n\in\w}\subseteq K\subseteq P(K)$ of pairwise distinct Dirac measures, converging to the Dirac measure $\mu_0$.
The choice of the transfinite sequence $\big((\mu_n^\alpha\big)_{n\in\w})_{\alpha\in\w_1}$ guarantees that the set
\[
S:=\{\alpha\in\w_1:\;\mu_n^\alpha=P\pi^{\w_1}_\alpha(\mu_n)\; \mbox{ for all } \; n\in\w\}
\]
is stationary in $\w_1$, and hence it contains an infinite limit ordinal $\alpha\in S$ such that the function $\w\to 2^\alpha\subseteq P(2^\alpha)$, $n\mapsto P\pi_\alpha^{\w_1}(\mu_n)$, is injective. Then the ordinal $\alpha$ belongs to the set $\Lambda$ and the inductive condition (8) guarantees that $\bigcup_{n\in\w}\supp(\mu^\alpha_{2n+1})\subseteq A_\alpha$ and $\bigcup_{n\in\w}\supp(\mu^\alpha_{2n+2})\subseteq B_\alpha$. Since the sets $A_\alpha\times\{0\}$ and $B_\alpha\times\{1\}$ have disjoint closures in $X_{\alpha+1}$, the sequence $\big(P\pi^{\w_1}_{\alpha+1}(\mu_n)\big)_{n\in\w}$ is divergent, which contradicts the convergence of the sequence $(\mu_n)_{n\in\w}$.
\end{proof}

\begin{claim} \label{cl:Efimov}
The space $K$ is a simple crowded Efimov space.
\end{claim}

\begin{proof}
By a result of Koppelberg \cite{Kop} (whose topological proof is given by Dow \cite[Proposition~1]{Dow}), the compact space $K$, being simple,  does not admit continuous surjective maps onto $[0,1]^{\w_1}$. Since $\beta\w$ admits a continuous surjective map onto $[0,1]^{\w_1}$ (because $[0,1]^{\w_1}$ is separable), the Tietze--Urysohn extension theorem  implies that $K$ contains no copies of $\beta\w$. This result and  Claims~\ref{cl:simple} and \ref{cl:seq} imply that $K$ is a simple crowded Efimov space.
\end{proof}

\begin{claim}\label{cl:converge}
Every $c_0$-vanishing sequence $(\mu_n)_{n\in\w}$ in $P(K)$ has a subsequence that converges to some nonatomic measure.
\end{claim}

\begin{proof}  By Lemma~\ref{l:vanish}, the $c_0$-vanishing property of $(\mu_n)_{n\in\w}$ implies the existence of an ordinal $\upsilon\in\w_1$ such that for every countable ordinal $\alpha\ge\upsilon$ the sequence of measures $(P\pi^{\w_1}_\alpha(\mu_n))_{n\in\w}$ is $c_0$-vanishing.
By the choice of the transfinite sequence $\big((\mu_n^\alpha)_{n\in\w}\big)_{\alpha\in\w_1}$, the set
\[
S=\{\alpha\in\w_1:\;\mu_n^\alpha=P\pi^{\w_1}_\alpha(\mu_n)\; \mbox{ for all } n\in\w\}
\]
is stationary and hence contains a limit ordinal $\alpha\in S$ such that $\alpha\ge \upsilon$. By the choice of $\upsilon$, the sequence of measures $(\mu_n^\alpha)_{n\in\w}=\big(P\pi_\alpha^{\w_1}(\mu_n)\big)_{n\in\w}$ is $c_0$-vanishing and hence $\alpha\in \Omega$.
Let $\mu\in P(K)$ be any accumulation point of the sequence $(\mu_n)_{n\in I_\alpha}$. Then $P\pi_{\alpha+\w}^{\w_1}(\mu)$ is an accumulation point of the sequence $\big(P\pi_{\alpha+\w}^{\w_1}(\mu_n)\big)_{n\in I_\alpha}$.
Since the ordinal $\alpha$ belongs to the set $\Omega$, the inductive condition (9) guarantees that the measure $P\pi_{\alpha+\w}^{\w_1}(\mu)$ is nonatomic and so is the measure $\mu$. By Lemma~\ref{l:first-countable}, the space $P(K)$ is first-countable at $\mu$. Then there exists an infinite set $J\subseteq I_\alpha$ such that the sequence $(\mu_k)_{k\in J}$ converges to $\mu$.
\end{proof}

Claim~\ref{cl:converge} and Lemma~\ref{l:Pna-norm} imply

\begin{claim}
Every sequence in $P(K)$ that converges to an atomic measure $\mu\in P_a(K)$ converges to $\mu$ in norm.
\end{claim}

\begin{claim} \label{cl:Pna-ssp}
The space $P_{na}(K)$ is  non-metrizable, first-countable, \v{C}ech-complete, sequentially compact, and the set $P_{na}(K)$ is dense in $P(K)$.
\end{claim}

\begin{proof}
The first-countability, \v Cech-completeness, and density of $P_{na}(K)$ in $P(K)$  follows from Lemmas~\ref{l:first-countable} and \ref{l:Cech}, respectively. Claim~\ref{cl:converge} implies that the space $P_{na}(K)$ is sequentially compact.

To show that $P_{na}(K)$ is  non-metrizable, suppose for a contradiction that $P_{na}(K)$ is metrizable. Then  $P_{na}(K)$ being also sequentially compact must be compact. Since, by Lemma \ref{l:Cech}, $P_{na}(K)$  is dense in $P(K)$ we obtain that $P_{na}(K)=P(K)$, a contradiction. Thus $P_{na}(K)$ is  not metrizable.
\end{proof}

\begin{claim} \label{cl:Efimov-ssp}
The space $P(K)$ is selectively sequentially pseudocompact but not sequentially compact.
\end{claim}

\begin{proof}
The space  $P(K)$ is selectively sequentially pseudocompact by Lemma \ref{l:dense-ssp} and Claim \ref{cl:Pna-ssp}.  Since $K$ is a closed  non-sequentially compact subspace of $P(K)$, the space $P(K)$ is also  non-sequentially compact.
\end{proof}

\begin{claim}
The Banach space $C(K)$ has the Gelfand--Phillips property.
\end{claim}

\begin{proof}
Corollary 2.2 of \cite{BG-GP-Banach} states that if $P(K)$ is selectively sequentially pseudocompact, then $C(K)$ has the Gelfand--Phillips property. It remains to note that $P(K)$ is selectively sequentially pseudocompact by Claim \ref{cl:Efimov-ssp}.
\end{proof}

\section{Acknowledgements}

The first author expresses his sincere thanks to Grzegorz Plebanek for an inspiring discusion on the construction of a simple Efimov space from his joint paper with  Mirna D\v{z}amonja \cite{DP}.



\begin{thebibliography}{}

\bibitem{Ban1}
T. Banakh, {\em The topology of spaces of probability measures, I: functors $\mathcal{P}_\tau$ and $\hat{P}$.} Matematychni Studii \textbf{5} (1995), 65--87; (English transl.: {\tt arxiv.org/pdf/1112.6161.pdf}).

\bibitem{Ban2}
T. Banakh, {\em The topology of spaces of probability measures, II: barycenters of probability Radon measures and metrization of the functors $\mathcal{P}_\tau$ and $\hat{P}$.} Matematychni Studii \textbf{5} (1995), 88-106; (English transl.: {\tt arxiv.org/pdf/1206.1727.pdf}).

\bibitem{BG-GP-Banach}
T. Banakh, S. Gabriyelyan, {\em Banach spaces with the (strong) Gelfand--Phillips property}, preprint\\ ({\tt arxiv.org/abs/2110.07991}).






\bibitem{Bog-18}
V.I.~Bogachev, \emph{Weak convergence of measures}, AMS, Mathematical survays and monographs v. 234, 2018.











\bibitem{DAS1}
A.~Dorantes-Aldama, D.~Shakhmatov, {\em Selective sequential pseudocompactness}, Topology Appl. {\bf 222} (2017), 53--69.

\bibitem{DrewEm}
L. Drewnowski, G. Emmanuele, \emph{On Banach spaces with the Gelfand--Phillips property, II}, Rend. Circ. Mat. Palermo \textbf{38} (1989), 377--391.


\bibitem{Dow}
A. Dow, \emph{ Efimov spaces and the splitting number}, Topology Proc. \textbf{29}:1 (2005), 105--113.

\bibitem{DM}
A.~Dow, R.~Pichardo-Mendoza, \emph{Efimov spaces, CH, and simple extensions}, Topology Proc. \textbf{33} (2009), 277--283.
	
\bibitem{DP}
M.~D\v zamonja, G.~Plebanek, \emph{On Efimov spaces and Radon measures}, Topology Appl. \textbf{154}:10 (2007), 2063--2072.

\bibitem{Eng}
R.~Engelking, \emph{General Topology}, Heldermann Verlag, Berlin, 1989.


\bibitem{Fed75}
V.V.~Fedorchuk, {\em A bicompactum whose infinite closed subsets are all $n$-dimensional}, (Russian) Mat. Sb. (N.S.) {\bf 96(138)} (1975), 41--62 (in Russian).

\bibitem{Fed76}
V.V.~Fedorchuk, {\em Completely closed mappings, and the compatibility of certain general topology theorems with the axioms of set theory}  (Russian) Mat. Sb. (N.S.) {\bf 99} (141) (1976), 3--33.

\bibitem{Fed77}
V.V. Fedor\v{c}uk, \emph{A compact space having the cardinality of the continuum with no convergent sequences}, Math. Proc. Cambridge Phil. Soc. \textbf{81} (1977), 177--181.

\bibitem{Fedorchuk-s}
V.V. Fedorchuk, \emph{Probability measures in Topology}, Uspehi Mat. Nauk \textbf{46} (1991) 41--80 (in Russian); English transl.: Russian Math. Surveys \textbf{46} (1991), 45--93.


\bibitem{Jensen} R.~Jensen, {\em The fine structure of the constructible hierarchy. With a section by Jack Silver.} Ann. Math. Logic 4 (1972), 229--308.







\bibitem{Hart}
K.P.~Hart, {\em Efimov's problem}, in: {\em Open Problems in Topology, II} (E.~Pearl, ed.), 171--177, Elsevier, 2007.

\bibitem{Haydon-72}
R. Haydon, \emph{On dual $L^1$-spaces and injective bidual Banach spaces}, Isr. J. Math. \textbf{31} (1972), 142--152.

\bibitem{Jar}
H.~Jarchow, \emph{Locally Convex Spaces}, B.G. Teubner, Stuttgart, 1981.




\bibitem{Kop}
S.~Koppelberg, \emph{ Minimally generated Boolean algebras}, Order, \textbf{5}:4 (1989),  393--406.

\bibitem{Linden}
J. Lindenstrauss, \emph{Weakly compact sets, their topological properties and Banach spaces they generate}, Proc. Symp. Infinite Dim. Topology, 1967, Annals of Math. Studies, Princeton Univ. Press, 1972.


\bibitem{Nyikos}
P.J.~Nyikos, {\em Classic problems—25 years later. I.} Proceedings of the Spring Topology and Dynamical Systems Conference (Morelia City, 2001). Topology Proc. {\bf 26}:1 (2001/02), 345--356.






\bibitem{Pol}
R.~Pol, \emph{ Note on the spaces of regular probability measures whose topology is determined by countable subsets}, Pacific J. Math. \textbf{100}  (1982), 185--201.


\bibitem{Schlumprecht-Ph} T. Schlumprecht, \emph{Limited sets in Banach spaces}, Dissertation, Univ. Munich, 1987.

\bibitem{TZ} A.~Teleiko, M.~Zarichnyi, {\em Categorical Topology of compact Hausdorff spaces}, VNTL Publ., Lviv, 1999.




\end{thebibliography}
\end{document}